\documentclass[12pt]{amsart}
\usepackage{amsmath, amssymb, euscript,  enumerate}
\usepackage{color}
\usepackage{epsfig,graphicx}
\usepackage{ulem}
\usepackage{cite}
\usepackage{pxfonts}
\usepackage{color,wrapfig}

\renewcommand{\l}{\langle}
\renewcommand{\r}{\rangle}
\renewcommand{\phi}{\varphi}
\newcommand{\eps}{\epsilon}
\newcommand{\Z}{\mathbb{Z}}

\newcommand{\close}[1]{\overline{#1}}
\newcommand{\weak}{\rightharpoonup}

\newcommand{\supp}{\operatorname{supp}}
\newcommand{\Llra}{\Longleftrightarrow}

\newcommand{\R}{{\mathbb R}}

\newcommand{\BN}{{\mathbb N}}

\newcommand{\BR}{{\mathbb R}}

\newcommand{\BZ}{{\mathbb Z}}

\newcommand{\cA}{{\mathcal A}}

\newcommand{\cC}{{\mathcal C}}

\newcommand{\cK}{{\mathcal K}}

\newcommand{\cT}{{\mathcal T}}

\newcommand{\cZ}{{\mathcal Z}}

\newcommand{\ra}{\rightarrow}
         \newcommand{\e}{\varepsilon}

   \newcommand{\cpt}{\text{cap}}

\newtheorem{thm}[subsubsection]{Theorem}
\newtheorem{lemma}[subsubsection]{Lemma}

\newtheorem{remark}[subsubsection]{Remark}

\newtheorem{example}[subsubsection]{Example}
\newtheorem{definition}[subsubsection]{Definition}

\numberwithin{equation}{section}

\begin{document}
\title{Highly Oscillating Thin Obstacles}

\author{Ki-ahm Lee}
\address{Seoul National University, Seoul, 151-747, Korea
\& Korea Institute for Advanced Study, Seoul,130-722, Korea}
\email{kiahm@snu.ac.kr}
\author{Martin Str\"omqvist}
\address{Royal Institute of Technology, Stockholm, SE-100 44, Sweden}
\email{stromqv@math.kth.se}
\author{Minha Yoo}
\address{Seoul National University, Seoul, 151-747, Korea}
\email{minha00@snu.ac.kr}
\thanks{We would like to thank  Henrik Shahgholian  for  his helpful comments. Ki-ahm Lee has been supported by the Korea-Sweden Research Cooperation Program.
This project is part of an STINT (Sweden)-NRF (Korea) research cooperation program.
}
\begin{abstract}
The focus of this paper is on a thin obstacle problem where the obstacle is defined on the intersection between a
hyper-plane $\Gamma$ in $\R^n$
and a periodic perforation $\cT_\e$ of $\R^n$, depending on a small parameter $\e>0$. As $\e\to 0$,
it is crucial to estimate the frequency of intersections and to determine this number locally.
This is done using strong tools from uniform distribution.
By employing classical estimates for the discrepancy of sequences of type $\{k\alpha\}_{k=1}^\infty$, $\alpha\in\R$, 
we are able to extract rather precise information about the set $\Gamma\cap\cT_\eps$.
As $\e\to0$, we determine the limit $u$ of the solution $u_\e$ to the obstacle problem
in the perforated domain, in terms of a limit equation it solves. We obtain the typical "strange term" behaviour for
the limit problem, but with a different constant taking into account the contribution of all different intersections,
that we call the averaged capacity.
Our result depends on the normal direction of the plane, but holds for a.e. normal on the unit sphere in $\R^n$.
\end{abstract}

\maketitle

\section{Introduction}\label{sec-intro}
\subsection{Formulation of the problem}
We consider the thin obstacle problem in a class of perforated domains.
For $\e>0$ we construct a perforated domain $\Gamma_\e$ as follows.
Let $Q_\e=(-\e/2,\e/2)^n$ and let $Q_\e(x)=x+(-\e/2,\e/2)^n$.
Note that the cubes $Q_\e(\e k)$ for $k\in \BZ^n$ are disjoint and
\[
\bigcup_{k\in\Z^n} \overline{Q_\e(\e k)}=\R^n.
\]
Next we perforate each cube by a
small hole: Let $T$ be compact subset of the unit ball $B_1$ with Lipschitz boundary $\partial T$, and for $a_\e<\e/2$ and
$k\in\Z^n$, define $T_\e=a_\e T$ and $T_\e^k = a_\e T+\e k$. The set
\[
\mathcal{T}_\e = \bigcup_{k\in\Z^n}T_\e^k
\]
is to be thought of as a periodic background in the problem.

Let $\Omega$ be a domain in $\R^n$, and let $\Gamma = \Gamma_\nu $ be a hyper plane with surface measure $\sigma$, defined by
\begin{equation} \label{gamma}
\Gamma_\nu = \{x \in \R^n : x \cdot \nu = x^0 \cdot \nu \}
\end{equation}
for given $\nu \in S^{n-1}$ and $x^0 \in \R^n$.

The set
\[
\Gamma_\e = \Gamma \cap\left(\bigcup_{k\in \BZ^n} T_\e^k\right)
\]
describes the intersection between the hyper-plane and the periodic background.
Then, for a given $\psi\in L^\infty(\Omega)\cap H^1(\Omega)$ such that $\psi\le 0 $ on $\partial\Omega$, we define the obstacle
\[
\psi_\e=\psi\chi_{\Gamma_\e}=
\left\{\begin{array}{ll}
\psi(x)&\text{if }x\in\Gamma_\e,\\
0&\text{if }x\not\in\Gamma_\e,
\end{array}\right.
\]
and the admissible set
\begin{equation}
\cK_{\psi_\e}=\{v\in H_0^1(\Omega):v\ge\psi_\e\}.
\label{eq-K}
\end{equation}
The inequality in \eqref{eq-K} is to be interpreted in the sense of trace, i.e.
$\text{Trace}_{\Gamma_\e}(u_\e-\psi)\ge0$ on $\Gamma_\e$ and $u_\e\ge0$ a.e. in $\Omega\setminus\Gamma_\e$.
We consider the following thin obstacle problem, for $f\in L^2(\Omega)$:
\begin{equation}
\left\{\begin{aligned}
&\int_\Omega\nabla u_\e\cdot\nabla(v-u_\e)dx\ge \int_\Omega(v-u_\e)fdx,\quad\text{for all }v\in\cK_{\psi_\e},\\
&u_\e\in\cK_{\psi_\e}.
\end{aligned}\right.
\label{eq-main-e}
\end{equation}
The variational inequality \eqref{eq-main-e} has a unique solution $u_\e\in\cK_{\psi_\e}$ which can be obtained
as the unique minimizer of the strictly convex and coercive functional
\[
J(v):=\int_\Omega\frac12|\nabla v|^2-fv dx, \quad v\in\cK_{\psi_\e}.
\]
We refer to Evans \cite{MR2597943} for the definition of trace and for the above minimization problem.



As $\e\to0,$ we are interested in the asymptotic behaviour of
$u_\e$. We want to determine $u=\lim_{\e\to0}u_\e$ in terms of
an effective equation that it solves. The procedure of finding the
effective equation, that does not depend on any microstructure in
$\Omega$, is called homogenization.

\subsection{Related Works}

In \cite{ShaLee}, Lee and Shahgholian study the Diriclet problem in a domain $\Omega$ with oscillating boundary data.
The boundary data is the restriction to $\partial\Omega$ of a function $g_\e$ that is $\e$ - periodic in $\BR^n$.
The common feature of that problem and the present is that the asymptotic behaviour is very sensitive with respect to
the normal field of the boundary, or in this case, the normal of the hyper-plane.

Obstacle problems in perforated domains, i.e. obstacle problems where the obstacle is given by
\[
\psi_\e= \psi\chi_{\mathcal{T}_\e}
\]
for some given $\mathcal{T}_\e$, have been studied extensively.
A common structure of the set describing the perforations $\mathcal{T}_\e$ is
\[
\mathcal{T}_\e=\bigcup_{k\in\Z^n}T_\e^k
\]
for some given set $T$ and $a_\e=o(\e)$, or a periodic distribution of holes on a hyper-surface in $\Omega$.
The paper \cite{MR1493040} by Cioranescu and Murat is a standard reference for these problems and the framework
developed therein
includes the hyper-surface case. 
Other interesting references for perforated domains include \cite{MR695419}, \cite{MR2504035}, \cite{MR2440878}.

The novelty of this paper is that the perforated surface $\Gamma_\e$ does not have a lattice structure in
the sense that the perforations are not evenly spaced, and this introduces a substantial difficulty.
The approach taken in this paper is based on the energy method where the construction of correctors is essential,
see section ~\ref{outline}. Our main reference for this is \cite{MR1493040}.

\subsection{Main Theorem}

To describe the effective equation for $u=\lim_{\e\to0}u_\e$, we introduce the averaged capacity,
depending on a direction $\nu$. First we recall the usual capacity of a subset of $\BR^n$, $A\subset B_1$ in case $n=2$.
\begin{definition}\label{capacity}
If $A$ is a compact subset of $\BR^n$, the capacity of $A$, denoted $\cpt (A)$, is
\[
\text{cap}(A) = \inf\left\{\int_{\R^n}|\nabla\phi|^2dx:\phi\in C_c^\infty(\R^n),\;\phi\ge 1\text{ on }A\right\}, \;\text{ if }n\ge 3,
\]
and
\[
\text{cap}(A) = \inf\left\{\int_{B_1}|\nabla\phi|^2dx:\phi\in C_c^\infty(B_1),\;\phi\ge 1\text{ on }A\right\}, \;\text{ if }n=2.
\]
\end{definition}
There are several ways of extending the capacity to non-compact sets, see for example \cite{MR1158660}.
\begin{definition}\label{def-avcap}[Averaged Capacity]
Suppose $\Gamma$ is a hyper plane in $\R^n$ with normal $\nu\in S^{n-1}$ and
define the family of hyper planes
\[
\Gamma_\nu(s) := \Gamma + s \nu,\quad s\in\R.
\]
If $T\subset \R^n$ and
\begin{equation}\label{f}
f(s)=\cpt(T\cap\Gamma_\nu(s))
\end{equation}
is integrable, we set
\[
\cpt_\nu(T):=\int_{-\infty}^\infty f(s)ds
\]
and call this quantity the \emph{averaged capacity} of $T$ with respect to $\nu$.
The set $T\cap\Gamma_\nu(s)$ is illustrated in Figure 1.
\end{definition}

\begin{thm}\label{thm-main}
Assume $n\ge 3$ and for a given $\nu\in S^{n-1}$ and $x^0\in\R^n$, 
let $\Gamma$ be the hyper-plane defined in \eqref{gamma}. Let $u_\e$ be the solution to
\eqref{eq-main-e} and set $a_\e = \e^{\frac{n}{n-1}}$. Then, for a.e. $\nu\in S^{n-1}$,
$u_\e\weak u$ in $H_0^1(\Omega)$ where $u$ is the unique minimizer of
\begin{equation}\label{functional-main-thm}
J_\nu(v):= \int_\Omega\frac12|\nabla v|^2 -fvdx + \frac12\cpt_\nu(T)\int_\Gamma((\psi-v)^+)^2d\sigma,\quad v\ge 0.
\end{equation}

In particular, $u$ is the solution of
\begin{equation}\label{eq-main-thm}
-\Delta u = \cpt_\nu(T)(\psi-u)^+d\sigma + f\chi_{\{u>0\}}.
\end{equation}
\end{thm}

\begin{remark}
It is interesting to consider the case
when $\Gamma$ a more general hyper-surface, for example a piece of a sphere or cone. In chapter \ref{sec-surface}, we prove Theorem \ref{thm-surf} which is similar to Theorem \ref{thm-main} but only valid in dimension $n \ge 5$ and when $\Gamma$ satisfies the condition \ref{smallcap}. 
We are able to apply theorem \ref{thm-surf} when $\Gamma$ is a cylinder, see example \ref{ex-cylinder}, but we cannot
verify its hypothesis when $\Gamma$ is a piece of a sphere or a cone. This would require much more delicate error estimates of discrepancy, and remains an interesting problem.
\end{remark}

\begin{figure} \label{fig-1}
\includegraphics[width=130mm]{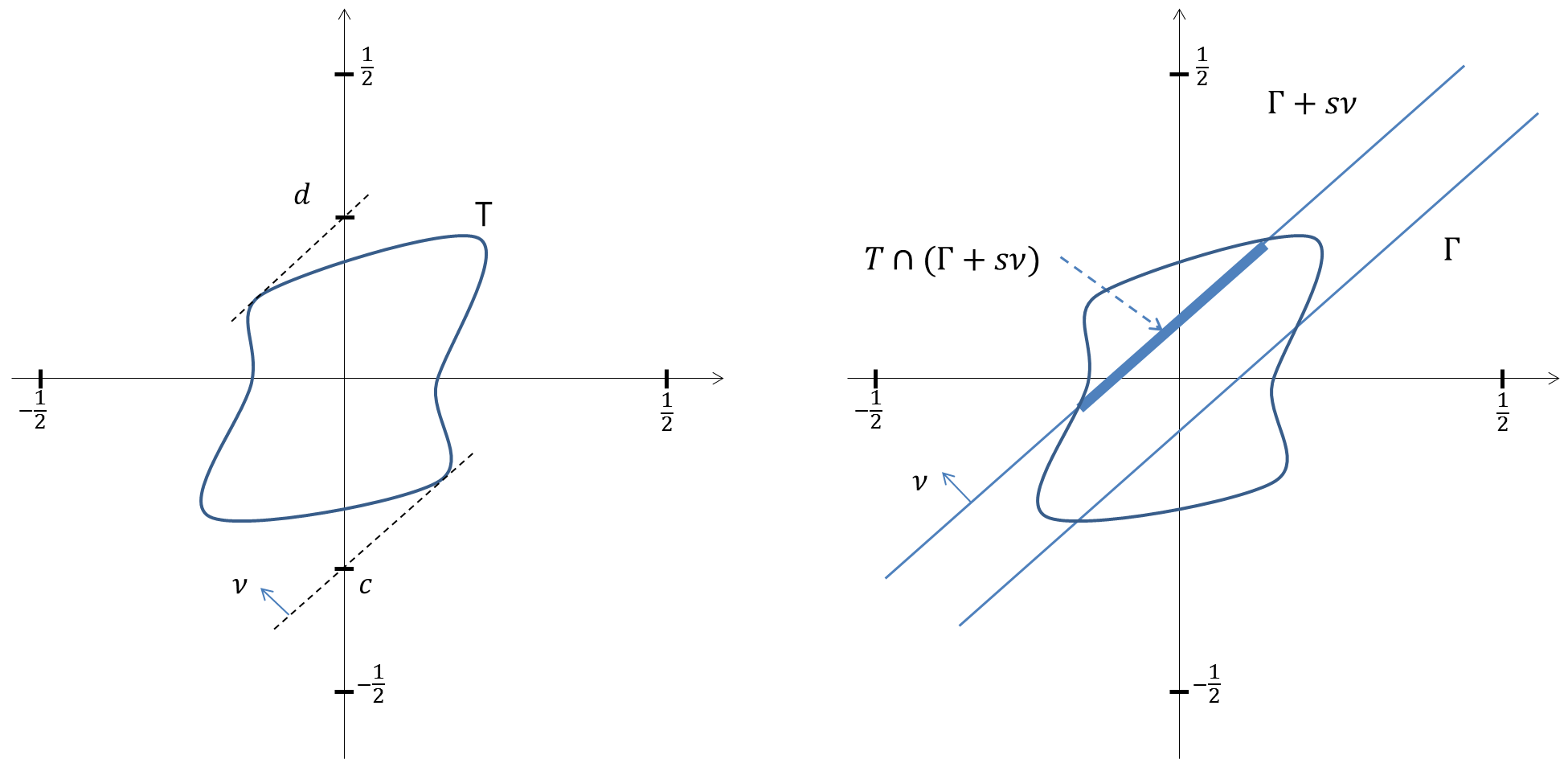}
\caption{The shape of $T$ and $T \cap (\Gamma + s \nu)$.}
\end{figure}

\subsection{Outline of the paper} \label{outline}
Our proof can be divided into two parts.
\begin{itemize}
\item
Local estimate of $\#(\Gamma\cap\cT_\e)$.

Since the sets $T_\e^k$ are located close to points $\e k$, $k\in\BZ^n$,
we need to understand how often the hyper plane $\Gamma$ intersects a certain neighborhood of
$\e k$, for all $\e k\in \Omega$.
To localize this, consider a set $E\subset \text{proj}_{\R^{n-1}}\Omega\cap\Gamma $.
Then $\Gamma$ is close to a point $\e k = \e(k',k_n)$ above $E$ if the $x_n$ - coordinate in
$(\e k',x_n)\in\Gamma$ is close to $\e k_n$.
Thus we are led to study the distribution of the $x_n$ - coordinates of $\Gamma$ at points
$\e k'\in \e\BZ^{n-1}$. This is done in Section \ref{Uniform distribution} on uniform distribution,
where we recall some classical results and use them to prove the important Lemma~\ref{lemma-equidist}.
We prove that for a.e. normal direction $\nu$ of the plane $\Gamma$, this distribution is uniform,
up to a small error.

\item Construction of correctors.

Having control on the intersections we construct correctors $w_\e$ that satisfy some standard assumptions, see lemma \ref{lem-cor-main}.
The energy of the correctors, which is closely related to the capacity of the set $\Gamma_\e$,
has to be finite and this determines the critical rate of $a_\e$. Below in ~\ref{critical rate}
we give a heuristic explanation on how to determine $a_\e$ using uniform distribution.
In Section ~\ref{sec-cor} we develop further properties of the correctors and prove \ref{lem-cor-main}.

\end{itemize}

We remark that the character of the problem may change drastically if the normal of the plane is altered,
or if the plane is translated. For example, if $0\in T$, the plane $\Gamma=\{x_n=0\}$ intersects every $T_\e^k\subset \Omega$
for $k=(k',0)$, but any small change in the normal will create completely different intersections.
Also, for $\Gamma=\{x_n=c\}$ the number of intersections may be zero or very large depending on a choice
of subsequence $\e_j\to0$. However, our result is that the character of the problem is the same
for a.e. normal direction, and is translation invariant.

\subsection{List of Notations}

\[
\begin{array}{ll}

\Omega &\text{A bounded open subset of }\R^n,\; n\ge3.\\

|\cdot| & n\text{ - dimensional Lebesgue measure}. \\

\chi_E& \text{The characteristic function of the set }E. \\

H_0^1(\Omega) & \text{the closure of }C_c^\infty(\Omega) \text{ w.r.t. the norm }\\
        &\|u\|_{H_0^1(\Omega)}=\left(\int_\Omega|\nabla u|^2dx\right)^{\frac12}. \\

Q_\e(\e k)&=(-\e/2,\e/2)^n+\e k,\quad k\in\BZ^n.\\
a_\e & = \e^{\frac{n}{n-1}}. \\

T & \text{a compact subset of }B_1\text{ such that }\close{\text{int}(T)}= T\text{ and }\partial T\text{ is Lipschitz}.\\

T_\e & = a_\e T.\\

T_\e^k & = a_\e T+\e k,\quad k\in\BZ^n. \\

\cT_\e &= \bigcup_{k\in\BZ^n}T_\e^k.\\

\Omega_\e&\Omega\setminus\cT_\e.\\

\Gamma =\Gamma_\nu & \text{a hypersurface in }\BR^n\text{ with normal }\nu. \\

\sigma & \text{surface measure on }\Gamma.\\

\Gamma_\e & =\Gamma\cap\cT_\e.\\

\gamma_\e^k & = \Gamma\cap T_\e^k.\\

\cpt(A) & \text{the capacity of the set }A, \text{ see Definition}~\ref{capacity}. \\

\cpt_\nu(T) &\text{the averaged capacity of the set }T, \text{see Definition}~\ref{def-avcap}.\\

\Gamma_\Omega' & = \text{proj}_{\BR^{n-1}}\Omega\cap \Gamma - \text{the projection of }\Omega\cap \Gamma\text{ on }
\BR^{n-1}.\\

\cZ_\e & = \e^{-1}\Gamma_\Omega'\cap\BZ^{n-1}.\\

\# A & = \text{the number of elements of a finite set }A. \\

N(\e) & = \#\cZ_\e = \#\left(\e^{-1}\Gamma_\Omega'\cap\BZ^{n-1}\right).\\

A(\e^p,t) & = \#\{k'\in \cZ_\e:\alpha\cdot k'/\BZ\in (t,t+\e^p)/\BZ\}. \\

\end{array}
\]

\subsection{Heuristic arguments and computation of the critical rate}
\label{critical rate}

The proof relies on the construction of correctors similar to those of Cioranescu and Murat in \cite{MR1493040}.
We will prove the existence of a function $w_\e$, called corrector, that satisfies the properties in lemma \ref{lem-cor-main}.
Once this has been established our main theorem follows in a rather standard way, see Lemma~\ref{lem-11} - ~\ref{lem-13}.
The function $(\psi-u_\e)^+$ is used in place of $z_\e$, which is bounded if $\psi$ is.

\begin{figure} \label{fig-2}
\includegraphics[width=80mm]{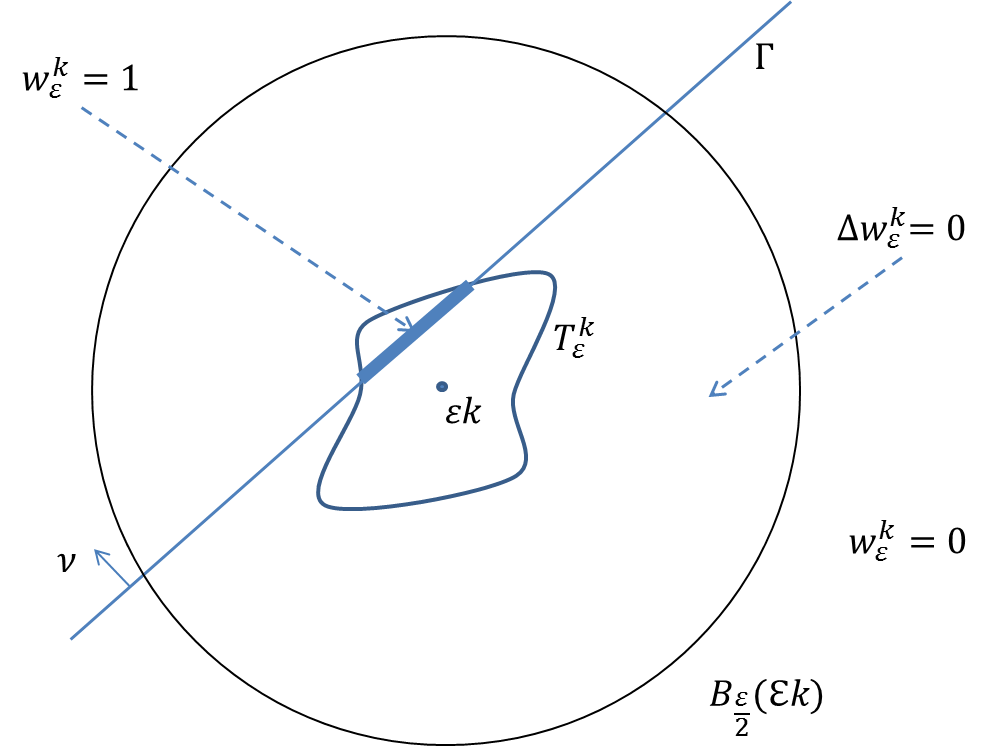}
\caption{}
\end{figure}

We obtain such $w_\e$ by defining $w_\e$ locally near the intersection between $\Gamma$ and $T_\e^k$, a component of
$\mathcal{T}_\e$. Suppose $a_\e<\e/2$ is a sequence whose decay rate is to be determined.
For
\[
\gamma_\e^k:=\Gamma\cap T_\e^k,
\]
we have diam$(\gamma_\e^k)=o(a_\e)$ and
\[
\gamma_\e^k = a_\e\left(a_\e^{-1}\Gamma\cap\e a_\e^{-1}k+T\right):=a_\e(\Gamma\cap(T+\text{translation})).
\]
We define
\begin{equation}\label{the_corrector}
w_\e=\sum_kw^k_\e,
\end{equation}
where $w_\e^k$ is the restriction of $w_\e$ to $Q_\e(\e k)$, given by
\begin{equation}\label{eq-cor-loc}
\begin{cases}
w^k_\e=1&\text{on }\gamma_\e^k\\
\Delta w^k_\e=0&\text{in }B_{\e/2}\setminus \gamma_\e^k\\
w^k_\e=0&\text{in }Q_\e\setminus B_{\e/2},
\end{cases}
\end{equation}
see Figure 2.
The energy of the correctors, i.e. the quantity
\[
\int_\Omega|\nabla w_\e|^2dx
\]
has to be uniformly bounded from above and below in order for lemma \ref{lem-cor-main} to hold.
We note that
\[
\int_\Omega|\nabla w_\e|^2dx=\sum_k\int_\Omega|\nabla w^k_\e|^2dx,
\]
and
\[
\int_\Omega|\nabla w^k_\e|^2dx\approx \cpt(\gamma_\e^k).
\]
Since
\[
\cpt(\gamma_\e^k) = \cpt(a_\e(\Gamma\cap(T+\text{translation}))) =
\left\{\begin{array}{l}
O(a_\e^{n-2})\text{ if }n\ge 3,\\
O((-\log a_\e)^{-1}) \text{ if }n=2,
\end{array}\right.
\]
we should have
\begin{equation}\label{eq-o0}
\int_\Omega|\nabla w_\e|^2dx\approx C\sum_k\cpt(\gamma_\e^k) =
\left\{\begin{array}{l}
CA_\e (a_\e)^{n-2},\quad n\ge 3,\\
CA_\e (-\log a_\e)^{-1},\quad n=2,
\end{array}\right.
\end{equation}
where $A_\e$ is the number of terms in the sum.

The above energy calculation tells us that the energy of $w_\e$ is related with the number of intersection points $A_\e$ between
$T_\e^k$ and $\Gamma$. So we need to estimate the size of $A_\e$. It is here that standard theory
of uniform distribution and discrepancy enters into the game.

To simplify the exposition we assume for the time being that
\[
\Omega=(0,1)^n\text{ and }\Gamma=\{x\cdot\nu=0\}.
\]
Suppose also $\nu_n\neq 0$ and that $\Gamma$ may be represented as
\[
\Gamma=\{(x',\alpha\cdot x'):x'\in (0,1)^{n-1}\}, \quad \alpha=(-\nu_1/\nu_n,\ldots,-\nu_{n-1}/\nu_n).
\]
To count the number of intersection points, we just need to consider
$k'\in \e^{-1}(0,1)^{n-1}\cap\BZ^{n-1}$.
Among those $k'$, whether $\Gamma$ and $T_\e^k$ intersect or not is determined by the $x_n$ - coordinate of $\Gamma$ at
$x'=\e k'$.
In fact, it is necessary that
\[
\e \left(\alpha\cdot k' -k_n \right) \in (c_\e, d_\e) = (a_\e c, a_\e d),\quad \text{for some }k_n\in\BZ,
\]
where $-\frac{1}{2} <c < d< \frac{1}{2}$
as indicated in Figure 3.
Note that for each $k'\in\e^{-1}(0,1)^{n-1}\cap\BZ^{n-1}$ there is a unique $k_n\in\BZ$ such that
$\alpha\cdot k'-k_n\in(-1/2,1/2]$, or equivalently, $\alpha\cdot k'/\BZ\in(-1/2,1/2]/\BZ$.

\begin{figure} \label{fig-3}
\includegraphics[width=80mm]{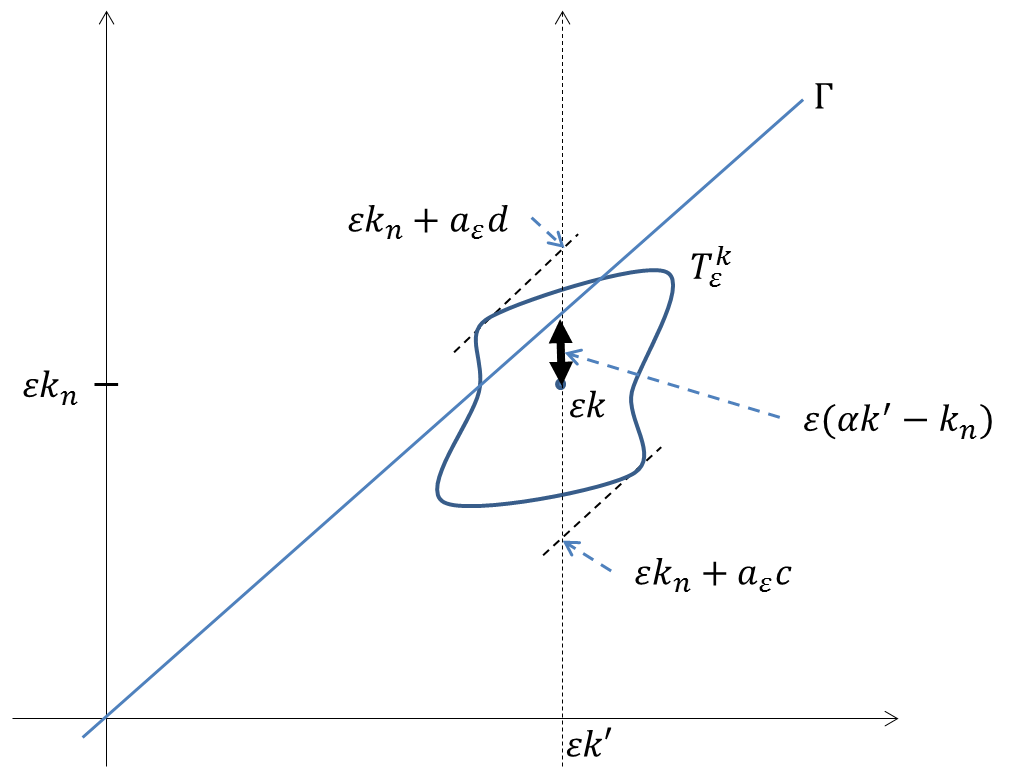}
\caption{}
\end{figure}

Actually, $\Gamma$ will intersect some $T_\e^k$ above $\e k'$ if and only if
\begin{equation}\label{eq-o1}
\alpha\cdot k'/\BZ\in (c_\e/\e,d_\e/\e)/\BZ.
\end{equation}
Hence
\begin{equation}\label{A_e}
A_\e= \#\{\alpha\cdot k'/\BZ\in (c_\e/\e,d_\e/\e)/\BZ:k'\in\cZ_\e\}.
\end{equation}
In equation \eqref{eq-o1} 
$k'$ ranges over the set
\[
\cZ_\e:=\{k'\in \BZ^{n-1}:\e k'\in (0,1)^{n-1}\},
\]
which contains $\e^{1-n}+o(\e^{2-n})$ points.
We see that the distribution (mod 1) of the sequence
\[
\{\alpha\cdot k'\}_{k'\in\cZ_\e}
\]
will determine the number of intersections. This distribution depends strongly on the arithmetic properties of
the components of $\alpha$, and thus of the normal $\nu$. However, we prove later in Section 2 that for a.e. $\nu\in S^{n-1}$
the sequence is rather "well" distributed. By this we mean that the fraction of points in the sequence $\{\alpha\cdot k'\}_{k'\in\cZ_\e}$
that intersect $\Gamma_\e$ equals the fraction that the interval $(c_\e,d_\e)$ occupies in the $\e$ cube, with some small error.
This is true as long as $d_\e-c_\e$ is not "too" small. That is, if we define
\begin{equation}\label{eq-o3}
N_\e=\#\cZ_\e\approx \e^{1-n},
\end{equation}
then
\begin{equation}\label{eq-o4}
\frac{A_\e}{N_\e}=\frac{d_\e-c_\e}{\e}+error,\quad error=o\left(\frac{d_\e-c_\e}{\e}\right)=o\left(\frac{a_\e}{\e}\right),
\end{equation}
provided $d_\e-c_\e$ is not too small. The error in \eqref{eq-o4} can be estimated by the discrepancy
(Definition~\ref{def_discrepancy}) of the sequence
\[
\{\alpha\cdot k':k'\in\cZ_\e\}.
\]
The smallest distance $h_\e$ in the normal direction between two parallel translations of $\Gamma$
that bound all intersections, see Figure 3, is related to $d_\e-c_\e$ as
\[
(d_\e-c_\e)e_n\cdot\nu=h_\e\Longleftrightarrow (d_\e-c_\e)=h_\e/\nu_n,
\]
and $h_\e=O(a_\e)$.
Using this in equation \eqref{eq-o3} and \eqref{eq-o4}
gives
\begin{equation}\label{eq-o5}
A_\e = O\left( N_\e\frac{a_\e}{\e}\right).
\end{equation}
Plugging this into \eqref{eq-o0} and using \eqref{eq-o3} yields, for $n\ge 3$,
\[
\int_\Omega|\nabla w_\e|^2dx\le CA_\e a_\e^{n-2}=C\e^{1-n}\frac{a_\e}{\e}a_\e^{n-2}=\frac{a_\e^{n-1}}{\e^n}.
\]
Also, a smaller fraction of the intersections $\gamma_\e^k$ will satisfy $\cpt(\gamma_\e^k)\ge ca_\e^{n-2}$, so we get a lower bound
\[
c\frac{a_\e^{n-1}}{\e^n}\le \int_\Omega|\nabla w_\e|^2dx.
\]
Thus, the choice
\begin{equation}
a_\e=\e^{\frac{n}{n-1}},\quad n\ge3,
\end{equation}
gives uniform lower and upper bounds on the energy of the correctors.

If $n=2$, the same argument as above, replacing $a_\e^{n-2}$ by $(-\log a_\e)^{-1}$ and recognizing that $N_\e=\e^{-1}$,
gives the condition
\begin{equation}\label{}
\lim_{\e\to0}\frac{-\e^{-3}a_\e}{-\log a_\e}=\text{constant},
\end{equation}
and this is true when
\begin{equation}\label{}
a_\e=-\e^{3}\log\e.
\end{equation}
However, in this case $d_\e-c_\e=O(a_\e)$ in \eqref{eq-o4} is too small, and this is why Theorem~\ref{thm-main}
is not valid in dimension $n=2$. Indeed, if the error in \eqref{eq-o4} is estimated by discrepancy
we get, using \eqref{ud:5},
\[
error \le \e^s,\;\text{ for any }s\in(0,1)\text{ and a.e. normal }\nu\in S^1,
\]
but this is not nearly enough since $a_\e/\e=-\e^2\log\e$ is much smaller.

The remaining properties in lemma \ref{lem-cor-main} will be proven in Section~\ref{sec-cor} on correctors.

\section{Correctors}\label{sec-cor}

We are going to construct the correctors $w_\e$ by determining the restriction of $w_\e$
to each cell $Q_\e(\e k)\subset\Omega$, $k\in\Z^n$.
Let
\begin{equation}\label{constr1}
\cC_\e^k=\left\{v\ge 1\text{ on } \gamma_\e^k, \;
v\in H_0^1(B_{\e/2})\right\},
\end{equation}
for any $\gamma_\e^k=\Gamma\cap T^k_\e$ such that $\gamma_\e^k \neq \emptyset$.
The solution $w_\e^k$ of equation \eqref{eq-cor-loc} can be characterized as follows:
\begin{equation}
\left\{
\begin{aligned}
\label{cor-4}
& w_\e^k\in \cC_\e^k,\\
&\int_{B_{\e/2}(\e k)}|\nabla w^k_\e|^2dx
=\inf\left\{\int_{B_{\e/2}(\e k)}|\nabla v|^2dx:\;v\in\cC_\e^k\right\}.
\end{aligned}
\right.
\end{equation}
This problem has a unique solution satisfying
$\Delta w_\e^k=0$ in $B_{\e/2}(\e k)\setminus\gamma_\e^k$ and $w_\e^k= 1$ on $\gamma_\e^k$,
and
\begin{equation}\label{cor-5}
\frac{\int_{B_{\e/2}(\e k)}|\nabla w^k_\e|^2dx}{\cpt(\gamma_\e^k)}\to 1,\;\e\to0.
\end{equation}
To see this we make a scaling and a translation
$x\mapsto a_\e x+\e k$, $\tilde w^k_\e(x)=w^k_\e(a_\e x+\e k)$.
Then \eqref{cor-5} becomes
\begin{equation}
\frac{\int_{B_{\e/2a_\e}}|\nabla \tilde w^k_\e|^2dx}{\cpt(\tilde \gamma^k)},
\end{equation}
where $\tilde \gamma^k=a_\e^{-1}(\gamma_\e^k-\e k)$ is independent of $\e$ and $\tilde w_\e^k$ satisfies
\begin{equation}
\label{cor-6}
\begin{aligned}
&\int_{B_{\e/2a_\e}}|\nabla \tilde w^k_\e|^2dx \\&=\inf\left\{\int_{B_{\e/2a_\e}}|\nabla v|^2dx:\;
v \in H_0^1(B_{\e/2a_\e})\text{ and } v\ge 1\text{ on } \tilde\gamma^k  \right\},
\end{aligned}
\end{equation}
which converges to $\cpt(\tilde \gamma^k)$.

We proceed with some proporties of averaged capacity, described in Definition
~\ref{def-avcap}, and its relation to the correctors. First we would like to point out that under the assumption 
that $T\subset\subset B_1$ has Lipschitz boundary, it is easy to check that the function $f(s)=\cpt(T\cap(\Gamma+s\nu))$ 
is continuous. 
Next we compute the averaged capacity for a ball.
\begin{example}
When $T$ is a ball we can compute the averaged capacity explicitly.
Say $T=B_r$. It is clear that $\cpt_\nu(B_r)$ is independent of $\nu$ so we assume $\nu=x_n$ and $\Gamma=\{x_n=0\}$.
Then
\[
\cpt_\nu(B_r)=\int_{-\infty}^\infty\cpt(B_r\cap(\Gamma+sx_n))ds.
\]
If $0\le s\le r$,
\[
B_r\cap(\Gamma+sx_n)=B_{\rho(s)}'+sx_n,\quad \rho(s)=\sqrt{r^2-s^2},
\]
and $\cpt(B_{\rho(s)}'+sx_n)=\cpt(B_{\rho(s)}')$ from the translation invariance of capacity.
We recall from Maz'ya, \cite{MR817985}, the capacity of the $(n-1)$- dimensional ball $B_\rho'$ with respect to $\R^n$, $n\ge 3$:
\[
\cpt(B_\rho')=\frac{\omega_n}{c_n}\rho^{n-2},
\]
where $\omega_n$ is the surface measure of the unit sphere in $\R^n$ and
\[
\left\{
\begin{array}{l}
c_3=\displaystyle\frac{\pi}{2},\quad c_4=1,\\
c_n=\displaystyle\frac{(n-4)!!}{(n-3)!!},\quad\text{if }n\ge 5\text{ is odd},\\
c_n=\displaystyle\frac{\pi(n-4)!!}{2(n-3)!!},\quad\text{if }n\ge 6\text{ is even}.
\end{array}\right.
\]
Thus,
\[
\cpt_\nu(B_r)=2\int_0^r\frac{\omega_n}{c_n}(r^2-s^2)^{\frac{n-2}{2}}ds=
2\frac{\omega_n}{c_n}r^{n-1}\int_0^1(1-\tilde{s}^2)^{\frac{n-2}{2}}d\tilde s\quad(\tilde{s}=rs).
\]
When $n=3$, this becomes, setting $\tilde{s}=\sin\tau$,
\[
\cpt_\nu(B_r)=2\frac{\omega_3}{c_3}r^{2}\int_0^{\pi/2}\cos^2\tau d\tau=2\frac{\omega_3}{c_3}r^{2}\frac{\pi}{4}
=\frac{\omega_3}{c_3}\frac{\pi}{2}r^{2}=\frac{4\pi/3}{\pi/2}\frac{\pi}{2}r^{2}=\frac{4\pi}{3}r^2.
\]

\end{example}

\begin{remark}\label{ex-cor}
If $T_\e=a_\e T$, $a_\e = \e^{\frac{n}{n-1}}$, then
\[
\e^{-n}\cpt_\nu(T_\e)=\cpt_\nu(T).
\]
This follows from the scaling properties of the capacity:
\[
f_\e(s):=\cpt(a_\e T\cap\Gamma_\nu(s)) = (a_\e)^{n-2}\cpt(T,\Gamma_\nu(s/a_\e))=(a_\e)^{n-2}f(s/a_\e).
\]
Thus
\[
\int f_\e(s)dt = (a_\e)^{n-2}\int f(s/a_\e)dt = (a_\e)^{n-1}\int f(s)ds = \e^n\int f(s)ds.
\]
\end{remark}
\begin{remark}\label{rem-cor}
If we set
\[
g_\e(s):=\int_{B_{\e/2}}|\nabla w^s_\e|^2dx,
\]
where $w^s_\e$ solves \eqref{cor-4} with  $\gamma^s_\e = T_\e \cap \Gamma_\nu( s)$,
then it can be concluded in the same way that
\[
g_\e(s)=(a_\e)^{n-2}G_\e(s/a_\e),
\]
where
\[
G_\e(s):=\int_{B_{\e/2a_\e}}|\nabla \tilde w^s_\e|^2dx,
\]
and $\tilde w^s_\e$ solves \eqref{cor-6} with $\gamma^s=T\cap\Gamma_\nu(s)$.
Moreover,
\[
\lim_{\e\to0}G_\e(s)=f(s),
\]
and according to the next lemma the convergence is uniform in $s$.
\end{remark}

\begin{lemma}
Let $\gamma \subset B_1(0)$ be any compact set. Then,
\begin{equation*}
\cpt(\gamma,B_R) := \int_{B_R} |\nabla W_R|^2 dx \ra \cpt(\gamma)
\end{equation*}
uniformly w.r.t. $\gamma$ where $W_R$ is the function satisfying
\begin{equation*}
\int_{B_R} |\nabla W_R|^2 dx  = \inf_{v \in K_R} \int_{B_R} |\nabla v|^2 dx, \quad K_R = \{ v \in H^1_0(B_R) ; v=1 \text{ on } \gamma \}.
\end{equation*}
\end{lemma}
\begin{proof}
Since capacity can be characterized by 
$\cpt(\gamma) = \inf_{v \in K} \int |\nabla v|^2 dx$, $ K = \{v\in H^1_0(\R^n) ; v \ge 1 \text{ on } \gamma \}$ and $K_R \subset K$,
\begin{equation} \label{cptinq_1}
\cpt (\gamma) \le \cpt(\gamma,B_R)
\end{equation}
holds from the definition for all $R >0$.

We can characterize the capacity by using the function $v \in H^1_0(\R^n)$ satisfying
\begin{equation*} \begin{cases}
\Delta v = 0 &\text{in } \R^n \setminus \gamma \\
v=1 &\text{on } \gamma \\
v=0 &\text{at infinity}.
\end{cases} \end{equation*}

From \cite{MR181436} page 27, we know that the capacity of $\gamma$ is given by
\begin{equation*}
\cpt(\gamma) = \int |Dv|^2 dx.
\end{equation*}
Let $ h(x) = \min \left\{1, \displaystyle\frac{1}{|x|^{n-2}} \right\}$.
Then, since $\Delta h =0$ in $\R^n \setminus B_1(0)$ and $h = 1$ in $B_1$,
\begin{equation*}
v \le h \text{ on } \R^n \setminus B_1
\end{equation*}
And hence we have
\begin{equation*}
v \le \displaystyle\frac{1}{R^{n-2}} =: M_R \text{ on } \R^n \setminus B_R.
\end{equation*}

Let $v_R(x) = \max\left\{0, \displaystyle\frac{v-M_R}{1-M_R}\right\}$. Then, $v_R$ is in $K_R$ and hence
\begin{equation} \label{cptinq_2} \begin{aligned}
\cpt(\gamma,B_R) &\le \int_{B_R \setminus \gamma} |\nabla v_R |^2 dx \\
&=\displaystyle\frac{1}{(1-M_R)^2} \int_{B_R \setminus \gamma} |\nabla v|^2 dx \\
&\le \displaystyle\frac{1}{(1-M_R)^2} \cpt(\gamma).
\end{aligned} \end{equation}

Finally, we get the conclusion by combining \eqref{cptinq_1} and \eqref{cptinq_2}.
\end{proof}

\begin{lemma}\label{corr.energy}
If $\nu\in S^{n-1}$ is such that $\nu^i/\nu^j\in \cA$, for at least one pair $(i,j)$, $i,j\in\{1,\ldots,n\}$,
then for any measurable subset $E$ of $\R^n$
\[
\int_E|\nabla w_\e|^2dx\to \sigma(\Gamma\cap E)\,\cpt_\nu(T).
\]
\end{lemma}
\begin{proof}
Without loss of generality we assume $\nu_{n-1}/\nu_{n}\in \cA$, i.e. $\alpha_{n-1}\in  \cA$,
where $\alpha\in\BR^{n-1}$ is given by \eqref{hp1}.
We may also assume $c=0$ in the representation $\Gamma=\{x\cdot\nu=c\}$, by Remark~\ref{rem-translation}.
Let
\[
\Gamma'_E= \text{proj}_{\BR^{n-1}} (E\cap\Gamma), \quad \cZ_\e=\e^{-1}\Gamma'_E\cap\BZ^{n-1}.
\]
Then
\begin{equation}\label{cor-11}
\sigma(\Gamma'_E)=\nu_n\sigma(E\cap\Gamma).
\end{equation}
Note that $(\e k + T_\e ) \cap \Gamma \neq \emptyset$ ($k=(k',k_n)$) is equivalent to
\begin{equation}\label{t}
\alpha k'/\BZ \in (c_\e/\e, d_\e/\e)/\BZ\Llra \alpha\cdot k'-k_n\in (c_\e/\e,d_\e/\e)
\end{equation}
for some constants $c_\e=a_\e c$ and $d_\e=a_\e d$ as described in Figure 3.
If \eqref{t} holds, let
\begin{equation}\label{tdef}
t=t(k')= \e (\alpha\cdot k'-k_n).
\end{equation}
Thus $t(k')=O(a_\e)$. Since
\begin{equation} \label{cor-12}
-\e k + \left(  (\e k + T_\e) \cap \Gamma  \right) = T_\e \cap \left( t(k')e_n + \Gamma \right),
\end{equation}
the shape of $(\e k + T_\e) \cap \Gamma $ is completely determined by $t=t(k')$.
Let $M$ be large positive integer and let $\delta=\displaystyle\frac{d - c}{M}$.

Define
\[
I(i) = I_M(i) = (c + (i-1) \delta, + i\delta),\quad i=1,\ldots,M,
\]
and let
\begin{equation*}
A_i(\e) = \#\{t(k')/a_\e \in I(i) : k' \in \cZ_\e \},~ i=1,\ldots, M.
\end{equation*}
Then, from lemma \ref{lemma-equidist}, we have
\begin{equation*}
A_i(\e) = (1+\rho(\e))N(\e)\frac{a_\e\delta }{\e},\quad \rho(0+)=0,
\end{equation*}
where $N(\e) = \# \cZ_\e$.

Since $\cup_{i=1}^M I(i) = (c, d)$, we have
\begin{equation*}
\int_E |\nabla w_\e |^2 dx = \sum_{i=1}^M \sum_{t(k')/a_\e \in I(i)} \int_{B_{\e/2}(\e k)} |\nabla w_\e^k |^2 dx,
\end{equation*}
where $w_\e^k$ is a solution satisfying \eqref{cor-4}. From \eqref{cor-12}, we see that the energy of $w_\e^k$ is the same as that of
$w_\e^{t \nu_n}$ in Remark~\ref{rem-cor}, since
\[
T_\e \cap \left( t(k')e_n + \Gamma \right) = T_\e \cap\left( \nu_nt(k')\nu + \Gamma \right).
\]
That is, 
\[
\int_{B_{\e/2}(\e k)} |\nabla w_\e^k |^2 dx = \int_{B_{\e/2}} |\nabla w^{t\nu_n}_\e |^2 dx=g_\e(t\nu_n).
\]
From this, we have the following:
\begin{equation*} \begin{aligned}
\int_E |\nabla w_\e |^2 dx &= \sum_{i=1}^M \sum_{t(k')/a_\e \in I(i)} \int_{B_{\e/2}(\e k)} |\nabla w_\e^k |^2 dx \\
&\le \sum_{i=1}^M A_i(\e) \sup_{t/a_\e \in I(i)} \int_{B_{\e/2}(\e k)} |\nabla w_\e^{\nu_n t} |^2 dx \\
&=\sum_{i=1}^M A_i(\e) \sup_{t/a_\e \in I(i)} g_\e(t\nu_n) \\
&\le (1+\rho(\e))N(\e)\frac{a_\e\delta}{\e} \sum_{i=1}^M \sup_{t \in I(i)} g_\e(a_\e t\nu_n) \\
&\le (1+\tilde \rho(\e))\frac{\sigma(\Gamma'_E)}{\e^{n-1}} \frac{a_\e\delta}{\e} \sum_{i=1}^M a_\e^{n-2} 
\sup_{t \in I(i)} G_\e(t\nu_n), \\
\end{aligned} \end{equation*}
where $\tilde\rho(0+)=0$.

Taking limit superior on both sides we obtain, by the uniform convergence of $G_\e$,
$$\limsup_{\e \ra 0} \int_E |\nabla w_\e |^2 dx \le
\sigma(\Gamma'_E)\frac{(d-c)}{M} \sum_{i=1}^M \sup_{t \in I(i)} f(\nu_n t).$$
Then, passing to the limit $M\to\infty$ and using \eqref{cor-11},
\begin{align*}
\limsup_{\e \ra 0} \int_E |\nabla w_\e |^2 dx& \le \sigma(\Gamma'_E) \int f(\nu_n t) dt \\
&=\frac{\sigma(\Gamma'_E)}{\nu_n} \int f(s) ds=
\sigma(\Gamma \cap E) \int f(s) ds.
\end{align*}

In a completely analogous way we find
\begin{equation*}
\liminf_{\e \ra 0} \int_E |\nabla w_\e |^2 dx \ge \sigma(\Gamma \cap E) \int f(s) ds.
\end{equation*}
Hence
\begin{equation*}
\liminf_{\e \ra 0} \int_E |\nabla w_\e |^2 dx = \sigma(\Gamma \cap E) \int f(s) ds,
\end{equation*}
as claimed.
\end{proof}

We proceed with some lemmas that are needed later for the proof of Lemma~\ref{lem-cor-main}.
Without loss of generality we may assume, by rotating the coordinates,
that $\Gamma=\{x_n=0\}$ and thus $d\sigma=dx'$.
\begin{lemma}[Compact embedding with $o(\e)$ error]\label{compact embedding}
Suppose $v_\e\weak v$ in $H_0^1(\Omega)$.
Then
\[
\frac{1}{\e}\int_0^\e\int_{\Gamma}(v_\e(x',x_n)-v(x',0))dx'dx_n\to 0.
\]
\begin{proof}
Since $\Gamma$ is part of the boundary of the set $\Omega_+=\Omega\cap \{x_n>0\}$,
$v_\e$ has a trace on $\Gamma$. That is, there is a continuous mapping
\[
H^1(\Omega)\mapsto H^{1/2}(\Gamma).
\]
For a definition of $H^{1/2}(\Gamma)$ and its properties, see \cite{MR1765047} p 51.
By the fundamental theorem of calculus, H\"older's inequality and Young's inequality:
\begin{align*}
\frac{1}{\e}\int_0^\e&\int_{\Gamma} (v_\e(x',x_n)-v_\e(x',0))dx'dx_n\\
&=\frac{1}{\e}\int_0^\e\int_{\Gamma \cap \Omega}\int_0^{x_n}\partial_{y_n}v_\e(x',y_n)dy_ndx'dx_n\\
&\le \frac{1}{\e}\int_0^\e\int_{\Gamma \cap \Omega}\left(\int_0^{x_n}dy_n\right)^{1/2}
\left(\int_0^{x_n}|v_\e(x',y_n)|^2dy_n\right)^{1/2}dx'dx_n\\
&\le \frac{1}{\e}\int_0^\e\int_{\Gamma \cap \Omega}\sqrt{x_n}\left(\int_0^{x_n}|\partial_{y_n}v_\e(x',y_n)|^2dy_n\right)^{1/2}dx'dx_n\\
&\le \frac{1}{\e}\int_0^\e \int_{\Gamma \cap \Omega} x_n^{1/4}x_n^{1/4}\left(\int_0^{x_n}|\partial_{y_n}v_\e(x',y_n)|^2dy_n\right)^{1/2}dx'dx_n\\
&\le \frac{1}{\e}\int_0^\e\left(\frac12\sqrt{x_n}\sigma(\Gamma)+\frac C 2\sqrt{x_n}\|v\|_{H_0^1(\Omega)}^2\right)dx_n\\
&\le C\sqrt\e.
\end{align*}
Additionally,
\begin{align*}
&\frac{1}{\e}\int_0^\e\int_\Gamma (v_\e(x',0)-v(x',0))dx'dx_n\\
&=\int_{\Gamma \cap \Omega}(v_\e(x',0)-v(x',0))dx'\to 0\text{ as }\e \to0,
\end{align*}
since the inclusion $H^{1/2}(\Gamma \cap \Omega)\subset L^2(\Gamma)$ is compact.
\end{proof}
\end{lemma}

Next we note that there exist measures $\mu^k_\e$ and $\nu^k_\e$ such that
\begin{equation}\label{cor-10}
\Delta w^k_\e=\mu^k_\e-\nu^k_\e,\quad\supp\mu^k_\e\subset\partial B_\e(\e k),\;\supp\nu^k_\e\subset \gamma^k_\e.
\end{equation}
We define
\begin{equation}\label{cor-mu}
\mu_\e=\sum_k\mu^k_\e.
\end{equation}

\begin{lemma}\label{lem 10}
If $\nu_i/\nu_j\in \cA$ for some $i,j\in\{1,\ldots,n\}$, then
\[
\mu_\e\weak^* \cpt_\nu(T)\sigma,\quad \text{weakly star in the sense of measures}.
\]
That is,
\[
\l\mu_\e,\phi\r\to \cpt_\nu(T)\int_\Gamma\phi d\sigma,\quad\text{for all }\phi\in C_c^\infty(\Omega).
\]
\end{lemma}
\begin{proof}
It is clear that
\[
\int_\Omega\phi d\mu_\e\le \|\phi\|_{L^\infty}\int_\Omega d\mu_\e,\quad \phi\in C_c^\infty .
\]
From \eqref{cor-10} and the fact that $(1-w_\e)$ is zero on $\Gamma_\e=\cup_k\gamma^k_\e$
and $1$ on $\cup_k\partial B_\e(\e k)$ we get
\[
\int_\Omega d\mu_\e=\int_\Omega\Delta w_\e(1-w_\e)=\int_\Omega|\nabla w_\e|^2dx\le C.
\]
Thus $\mu_\e\weak^*\mu$ for a subsequence, where $\mu$ is a finite measure.
Since finite measures on $\R^n$ are regular, it is enough to determine
$\lim_\e\int_Ad\mu_\e$ for every open and every closed set $E\subset \Omega$.
Moreover, it is clear that $\supp\mu\subset \Gamma$.
By Lemma~\ref{corr.energy},
\[
\int_{ E} d \mu=\lim_\e\int_{E}d\mu_\e =\int_E |\nabla w_\e|^2 dx\to \cpt_\nu(T)\sigma(E\cap\Gamma),
\]
which proves the lemma.
\end{proof}

We are now in a position to prove the key lemma of the paper.
\begin{lemma} \label{lem-cor-main} \item
For a.e. $\nu\in S^{n-1}$, 
there exists a sequence of functions $w_\e$ satisfying
\begin{enumerate}
\item
$w_\e=1$ on $\Gamma_\e$.
\item
$w_\e \ra 0$ weakly in $H^1_0(\Omega)$
\item
For every sequence $z_\e\in H_0^1(\Omega)$ such that $z_\e = 0$ on $\Gamma_\e$, $\|z_\e\|_{L^\infty(\Omega)}\le C$
and $z_\e \ra z$ weakly in $H^1_0(\Omega)$
there holds
\begin{equation}\label{lapl.corr}
\lim_{\e\to0}\l\Delta w_\e, \phi z_\e \r_{H^{-1},H_0^1} \to \cpt_\nu(T)\int_\Gamma \phi zd\sigma ,
\end{equation}
for all $\phi \in \mathcal{D}(\Omega)$.
\end{enumerate}
\end{lemma}
\begin{proof}
By Lemma~\ref{aenormal} there exists a pair $(i,j)$, $i\neq j$ such that $\nu_i/\nu_j\in\cA$ for a.e. $\nu\in S^{n-1}$. 
The first property (1) is clear from the construction of $w_\e$, see \eqref{the_corrector}-\eqref{eq-cor-loc}.
By Lemma~\ref{corr.energy}, $w_\e$ is uniformly bounded in $H_0^1(\Omega)$.
Thus we can choose a subsequence $w_{\e_j}$ which converges to some $w_0$ weakly in $H_0^1(\Omega)$.
Since $w_{\e_j}$ converges $w_0$ strongly in
$L^2(\Omega)$, we can select a subsequence of $w_{\e_j}$ which converges to $w_0$ almost everywhere. But, $w_\e$ converges to $0$
except on $\Gamma$ and hence $w_0=0$.
Now, we will show $(3)$.
By applying Lemma~\ref{compact embedding} to the function $v_\e=z_\e\phi$, we see that
\begin{align*}
&\frac{1}{\e}\int_{-\e/2}^{\e/2}\int_{\Gamma \cap \Omega}|(z_\e\phi)(x',x_n)-(z\phi)(x',0)|dx'dx_n\\
&=\int_{-1/2}^{1/2}\int_{\Gamma \cap \Omega}|(z_\e \phi)(x',\e x_n)-(z\phi)(x',0)|dx'dx_n\to 0
\end{align*}
and thus
\[
(z_\e\phi)(x',\e x_n)=:v_\e(x',x_n)\to v(x',x_n):=(z\phi)(x',0),
\]
a.e. on $S:=(\Gamma \cap \Omega) \times (-1/2,1/2)$.
By Egoroff's theorem we can assert the existence of a set $S_\delta$
such that
\[
v_\e\to v \text{ uniformly on }S_\delta,\quad |S\setminus S_\delta|<\delta,
\]
for any $\delta>0$.
Upon rescaling we find:
\begin{align*}
&\text{There exists }\e_0>0\text{ such that }\\
&|(z_\e\phi)(x',x_n)- (z\phi)(x',0)|<\delta \text{ on }S_\delta^\e,\quad
|S^\e\setminus S_\delta^\e|<\e\delta, \quad\text{for all }\e<\e_0,
\end{align*}
where, for any set $E\in\R^n$, $E^\e=\{(x',\e x_n):(x',x_n)\in E\}$.
Note that since $\Delta w_\e=\mu_\e-\nu_\e$ with $\supp\nu_\e\subset \Gamma_\e$ and $z_\e=0$ on $\Gamma_\e$,
\[
\int_\Omega\Delta w_\e\phi z_\e dx=\int_\Omega\phi z_\e d\mu_\e.
\]
This allows us to compute
\begin{equation}\label{eq:33}
\left|\int_\Omega(z_\e\phi-z\phi)d\mu_\e\right|\le \int_{S_\delta^\e}
|z_\e\phi-z\phi|d\mu_\e + 2\|z\phi\|_{L^\infty}\int_{S^\e \setminus S_\delta^\e}
d\mu_\e.
\end{equation}

According to Lemma~\ref{lem 10} the first integral on the right hand side of \eqref{eq:33} is bounded by
\begin{equation}\label{eq:34}
\delta\int_\Omega d\mu_\e\to \delta \cpt_\nu(T) \sigma(\Gamma).
\end{equation}

For the other term, we may cover $S\setminus S_\delta $ by a countable union of
cubes $Q_i$ such that $\sum_i|Q_i|<2\delta$ and, say, $|Q_i|=\eta_i^n$.
Let $Q_i'=\{x':(x',x_n)\in Q_i\}$ and note
that $Q_i=Q_i'\times (x_n^0,x_n^0+\eta_i)$ and $Q_i^\e=Q_i'\times (\e
x_n^0,\e x_n^0+\e \eta_i)$, for some $x_n^0$. To estimate the second
integral in \eqref{eq:33} it is convenient to construct a barrier for
$\mu_\e$. Let $B_r\supset T$ and let $\overline{w}_\e$ satisfy
\[
\left\{\begin{array}{l}
\overline{w}_\e=1\quad\text{on }B_{ra_\e}\\
\overline{w}_\e=0\quad\text{on }\partial B_{\e/2}\\
\Delta \overline{w}_\e=0 \quad\text{in }B_{\e/2} \setminus B_{ra_\e}
\end{array}\right.
\]
(Note that $B_{ra_\e}\subset B_{\e/2}$ if $\e$ is small enough). Then
since $B_{ra_\e}\supset T_\e$, the maximum principle tells us that
$\overline{w}_\e\ge w_\e$ in $B_{\e/2}$. This implies that
\[
0\le -\frac{\partial w_\e}{\partial n}|_{\partial B_{\e/2}}\le
-\frac{\partial\overline{w}_\e}{\partial n}|_{\partial B_{\e/2}}\le
C\e^{1-n}(a_\e)^{n-2}.
\]
It follows that
\[
\int_Ad\mu_\e\le C\sum_{k}\int_{\partial B_{\e/2}(\e k)\cap
A}\e^{1-n}(a_\e)^{n-2}dS,
\]
for any measurable $A\subset \Omega$.
In particular,
\begin{equation}
\begin{aligned}\label{cube-est}
\int_{Q_i^\e}d\mu_\e&\le C\sum_k\int_{\partial B_{\e/2}(\e k)\cap
Q_i^\e}\e^{1-n}(a_\e)^{n-2}dS\\
&\le C'\eta_i\sum_k\int_{\partial
B_{\e/2}(\e k)\cap Q_i'\times(-\e/2,\e/2)}\e^{1-n}(a_\e)^{n-2}dS,
\end{aligned}
\end{equation}
where
\[
\int_{\partial B_{\e/2}(\e k)\cap
Q_i'\times(-\e/2,\e/2)}\e^{1-n}(a_\e)^{n-2}dS \le C\eta_i a_\e^{n-2}.
\]
This follows from the fact the $x_n$ axis is scaled by
$\eta_i$:
\begin{align*}
\text{Area}(\partial B_{\e/2}(\e k)\cap Q_i^\e)&=\text{Area}(\partial B_{\e/2}(\e k)\cap Q_i'\times(\e
x_n^0,\e x_n^0+\e\eta_i))\\
&\le \text{Area}(\partial B_{\e/2}(\e k)\cap Q_i'\times(-\eta_i\e/2,\eta_i\e/2))\\
&\le C\eta_i\:\text{Area}(\partial B_{\e/2}(\e k)\cap Q_i'\times(-\e/2,\e/2))\\
&\le C\eta_i \e^{n-1},
\end{align*}
where $\text{Area}(E)=\int_EdS$.
Moreover we know that, for small $\e$, the sum in \eqref{cube-est} has approximately
$\e^{-n}a_\e\eta_i^{n-1}$ terms, by Lemma~\ref{lemma-equidist}. Thus
$\int_{Q_i^\e}d\mu_\e\le C\eta_i^n=C|Q_i|$, $\e$ small. In conclusion,
\[
\int_{S^\e \setminus S_{\delta_\e}}
d\mu_\e\le C\sum_i|Q_i|\le 3C\delta.
\]
It follows that
\[
\left|\int_\Omega(z_\e\phi-z\phi)d\mu_\e\right|\le C\delta.
\]
Thus, according to \eqref{eq:33}-\eqref{eq:34}, it remains to prove
that
\[
\lim_{\e\to0}\int_\Omega z\phi(x',0) d\mu_\e
\]
exists.
Since $\phi z$ is a measurable function on $\Gamma$ there exists by
Lusin's theorem a set $\Gamma_\delta$ such that
$|\Gamma\setminus\Gamma_\delta|<\delta$ and $\phi z$ is continuous
on $\Gamma_\delta$. By extending the function $\phi z(x',0)$ and
$\Gamma $ to $\Omega$ by $\phi z(x',x_n)=\phi z(x',0)$, it follows
from Lemma~\ref{lem 10} that
\[
\lim_{\e\to0}\int_{\Gamma_\delta}\phi z
d\mu_\e=\cpt_\nu(T)\int_{\Gamma_\delta}\phi zd\sigma.
\]

Using \eqref{eq:33} and Lemma~\ref{lem 10} we obtain
\[
\lim_\e\int_\Omega z_\e\phi d\mu_\e = \cpt_\nu(T)\int_{\Gamma_\delta}z\phi d\sigma
+\lim_\e\int_{\Gamma\setminus\Gamma_\delta\times(-\e/2,\e/2)}z\phi d\mu_\e.
\]
Since the second term is $O(\delta)$ this completes the proof.
\end{proof}

\section{Proof of Theorem \ref{thm-main}}

Having established Lemma~\ref{lem-cor-main}, Theorem~\ref{thm-main} follows in a standard way.
The arguments are very similar to those in \cite{MR1493040}, and we will just indicate the necessary modifications
that have to be made. We always assume that the normal $\nu$ satisfies $\nu_i/\nu_j$ for some $i\neq j$, 
so that Lemma~\ref{lem-cor-main} may be applied. 
\begin{lemma}[l.s.c. of the energy]
\label{lem-11}
Let $z_\e\in H_0^1$ be a uniformly bounded sequence which is bounded uniformly in $\e$ satisfies $z_\e\weak z$ 
in $H_0^1$ and $z_\e=0$ on $\Gamma_\e$.
Then, we have
\[
\liminf\int_\Omega |\nabla z_\e|^2dx\ge \int_\Omega|\nabla
z|^2dx+\cpt_\nu(T)\int_{\Gamma}z^2d\sigma .
\]
\end{lemma}
\begin{proof}
Identical to that of \cite{MR1493040}.
\end{proof}
\begin{lemma}\label{lem-12}
Let $u_\e$ be the solution of equation \eqref{eq-main-e}. Then, we have the following estimate:
\begin{equation}
\begin{aligned}
\label{limsup}
&\limsup \int_\Omega\frac12|\nabla u_\e|^2-f u_\e dx\\
&\le \inf_{v\in H_0^1,\;v\ge 0}\int_\Omega\frac12|\nabla v|^2-fvdx+
\frac12\cpt_\nu(T)\int_\Gamma((\psi-v)^+)^2 d\sigma.
\end{aligned}
\end{equation}
\end{lemma}
\begin{proof}
Let  $v \in C_c^\infty(\Omega)$ and $v\ge 0$. Define
\[
v_\e=(w_\e-1)(\psi-v)^++(\psi-v)^-+\psi
\]
and let us prove that $v_\e\in\cK_{\psi_\e}$.
If $x\in\Gamma_\e$ and $v(x)>\psi(x)$, then $v_\e=-(\psi-v)(x)+\psi(x)=v(x)>\psi(x)$.
If $x\in\Gamma_\e$ and $v(x)\le\psi(x)$, then $v_\e(x)=\psi(x)$, since
$w_\e=1$ on $\Gamma_\e$. It remains to show that $v_\e\ge 0$ in $\Omega\setminus\Gamma_\e$.
If $\psi(x)<v(x)$,
$v_\e(x)=v(x)\ge 0$. If $\psi(x)\ge v(x)$, then
\[
v_\e(x)=w_\e(\psi-v)(x)+v(x)\ge 0.
\]
Note also that $v_\e\weak v$ weakly in $H_0^1(\Omega)$.
From the definition of $u_\e$ we have
\[
J(u_\e)\le J(v_\e).
\]
We refer to \cite{MR1493040} for the rest of the proof.
\end{proof}

\begin{lemma}\label{lem-13}
\[
\liminf \int_\Omega\frac12|\nabla u_\e|^2\ge \int_\Omega\frac12|\nabla u|^2+ \cpt_\nu(T) \int_\Gamma((\psi-u)^+)^2 d\sigma.
\]
\end{lemma}
\begin{proof}
We use the identity
\[
u_\e=-(\psi-u_\e)+\psi=-(\psi-u_\e)^++(\psi-u_\e)^-+\psi.
\]
Since $u_\e \ge \psi$ on $\Gamma_\e$, $(\psi-u_\e)^+ =0$ on $\Gamma_\e$. Now consider
$\int_\Omega|\nabla u_\e|^2dx$ and apply Lemma~\ref{lem-11} on the term $(\psi-u_\e)^+$.
\end{proof}

\begin{proof}[Proof of Theorem~\ref{thm-main}] 
Let $u_\e$ be the solution of \eqref{eq-main-e} and let $v\in H_0^1(\Omega)$ be any function such that $v\ge \psi^+$ in $\Omega$.
Then $v\in \cK_{\psi_\e}$ for all $\e$ and one easily obtains a uniform bound of
$\|u_\e\|_{H_0^1(\Omega)}$ by using $v$ in \eqref{eq-main-e}.
Thus $u_\e\weak u$ in $H_0^1(\Omega)$ for a subsequence. From Lemma \ref{lem-12} and Lemma \ref{lem-13},
\begin{equation*} \begin{aligned}
&\int_\Omega\frac12|\nabla u|^2 -fudx + \cpt_\nu(T)\frac12\int_\Gamma((\psi-u)^+)^2d\sigma\\
&\le\liminf_{\e\to0}\int_\Omega\frac12|\nabla u_\e|^2 -fu_\e dx
\le\limsup_{\e\to0}\int_\Omega\frac12|\nabla u_\e|^2 -fu_\e dx\\
&\le\int_\Omega\frac12|\nabla v|^2 -fvdx +\cpt_\nu(T)\int_\Gamma((\psi-v)^+)^2d\sigma,
\end{aligned} \end{equation*}
for all $v\in H_0^1(\Omega)$, $v \ge 0$.
This proves that $u$ minimizes $J_\nu$ over $\{v\in H_0^1(\Omega): v\ge 0\}$, from which the uniqueness of the limit follows.
Thus all subsequential limits agree and this implies that the entire sequence $\{u_\e\}_{\e}$ converges to $u$, 
weakly in $H_0^1(\Omega)$.

The fact that $u$ solves \eqref{eq-main-thm} follows from standard considerations in variational inequalities. 
\end{proof}

\section{The Case of General Hyper-Surfaces}\label{sec-surface}

In this section we consider again problem \eqref{eq-main-e}, but for a more general class of surfaces than hyper-planes.
Our assumptions are that $\Gamma$ is a hyper-surface in $\R^n$ of class $C^2$ with a unit normal-field $\nu(x)$
such that $\nu_n(x)\ge \lambda>0$. Thus
\begin{equation}\label{surf1}
\Gamma\cap\Omega = \{(x',h(x')):x'\in\Gamma_\Omega'\}, \quad\|h\|_{C^2(\Gamma_\Omega')}\le C,
\end{equation}
for some $h\in C^2(\Gamma_\Omega')$. We recall that $\Gamma_\Omega'=\text{proj}_{\BR^{n-1}}\Gamma\cap\Omega$.
As before we define
\begin{equation}\label{surf2}
\alpha(x)=(\alpha_1(x),\ldots,\alpha_{n-1}(x)),\quad\alpha_i=\nu_i/\nu_n.
\end{equation}
The idea is to locally approximate $\Gamma$ by its tangent-plane in a neigbourhood of each $x\in\Gamma\cap\Omega$.
This can be compared to the theory developed in \cite{ShaLee}.
The diameter of this neigbourhood has to be small enough in order for the tangent-plane to be close to
$\Gamma$, but still large enough in order for some averaging to occur.
This leads to two necessary conditions on these neigbourhoods. For any $x_i\in\Gamma$, let $Q_{r_\e}(x_i)$
be the cube of side $r_\e$ and center $x_i$, and denote by $\pi_\e(x_i)$ the restriction to $Q_{r_\e}(x_i)$
of the tangent-plane to $\Gamma$ at $x_i$.

The first condition comes from the fact that the distance between $\Gamma$ and $\pi_\e(x_i)$ has to small enough in
order for the intersections between $\pi_\e(x_i)$ and $\cT_\e$ to be the same as those between $\Gamma$
and $\cT_\e$, up to a small error.
Since the size of the perforations are of order $a_\e=\e^{n/(n-1)}$ and the distance between $\Gamma$ and $\pi_\e(x_i)$
in $Q_{r_\e}(x_i)$ is controlled by $Cr_\e^2$, according to \eqref{surf1} and Taylor expansion, it is necessary that
$r_\e^2=o(a_\e)$. If we assume $r_\e=\e^q$, then
\begin{equation}\label{surf3}
r_\e^2=o(a_\e)\text{ if and only if }q>\frac{n}{2(n-1)}.
\end{equation}

The second condition comes from the discrepancy of the sequence
\begin{equation}\label{eqom}
\omega_{r_\e}=\omega_{r_\e}(x_i)=\{k'\cdot\alpha(x_i):k'\in \e^{-1}Q_{r_\e}'(x_i)\cap\BZ^{n-1}\}.
\end{equation}
The cardinality of the set $\e^{-1}Q_{r_\e}'(x_i)\cap\BZ^{n-1}$ is
\begin{equation}\label{surf4}
N(\e)=\left(\frac{r_\e}{\e}\right)^{n-1}=\e^{(q-1)(n-1)}.
\end{equation}
We need to determine the number
\begin{equation}\label{surf5}
A(\e^p)=\#\{k'\in \e^{-1}Q_{r_\e}'(x_i)\cap\BZ^{n-1}:k'\cdot\alpha/\BZ\in [t,t+\e^p]/\BZ\},
\end{equation}
for $p=1/(n-1)$ and any $t\in (0,1)$, compare Section~\ref{sec-unif-2}, equation \eqref{mod1perforation}.
If we assume $\alpha_{i}\in\cA$ for some $1\le i\le n-1$, then an application of step 1 in Lemma~\ref{lemma-equidist}
leads to the estimate
\begin{equation}\label{surf6}
\left|\frac{A(\e^p)}{N(\e)}-\e^p\right|\le D_\e(x_i)=o\left(\frac{1}{\e^{(q-1)s}}\right),\quad\text{for any }s\in(0,1),
\end{equation}
where $D_\e(x_i)$ is the discrepancy of the sequence $\omega_{r_\e}(x_i)$ defined in \eqref{eqom}.
Solving \eqref{surf6} for $A(\e^p)$ we find
\begin{equation}\label{surf7}
A(\e^p)=N(\e)\e^p+N(\e)o(\e^{s(1-q)}).
\end{equation}
Clearly, this information has value only if $\e^{1-q}=o(\e^{p})$, or equivalently $p<1-q$. This leads to
the condition
\begin{equation}\label{surf8}
q<\frac{n-2}{n-1}.
\end{equation}

In conclusion, we should locally approximate the hyper-surface $\Gamma$ by its tangent-plane in cubes
$Q_{\e^q}(x_i)$ where $q$ has to satisfy both \eqref{surf3} and \eqref{surf8}, i.e.
\begin{equation}\label{surf9}
\frac{n}{2(n-1)}< q<\frac{n-2}{n-1}.
\end{equation}
This is possible if and only if $n\ge 5$.

\subsection{Effective equations}

The correctors $w_\e$ constructed in Section~\ref{sec-cor} are defined in precisely the same way for the hyper-surface case.
We shall start by characterizing those surfaces for which we are able to generalize our homogenization result,
Theorem~\ref{thm-main}.
Fix $s\in(0,1)$. Let
\begin{equation}\label{surf10}
\Gamma^j=\{x\in\Gamma:D_\e(x)\le j\e^{s(1-q)}\},
\end{equation}
where $D_\e(x)$ is the discrepancy of $\omega_{r_\e}(x)$. The sequence $\omega_{r_\e}(x)$ is defined in \eqref{eqom}.
If $\alpha_i(x)\in \cA$ for some $1\le i\le n-1$, then $x\in \Gamma^j$ for large enough $j$,
by \eqref{surf6}. In fact, $\sigma$ - a.e. $x\in \Gamma\cap\Omega$ belongs to $\Gamma^j$ for large enough $j$.
This can be proved in the same way as Lemma~\ref{aenormal}.
To prove a homogenization result for $\Gamma$, we need the following hypothesis on $\Gamma$:
\begin{equation}\label{smallcap}
\lim_{j\to\infty}\cpt((\Gamma\setminus\Gamma^j)\cap\Omega)=0.
\end{equation}

\begin{lemma}\label{lem-av-j}
Let $u^j_\e$ solve \eqref{eq-main-e} with $\Gamma^j$ in place of $\Gamma$. Then as $\e\to0$,
$u^j_\e\weak u^j$ where $u^j$ is the unique minimizer of
\begin{equation}\label{functional-av-j}
J_\nu^j(v):= \int_\Omega\frac12|\nabla v|^2 -fvdx + \frac12\int_{\Gamma^j}\cpt_{\nu(x)}(T)((\psi-v)^+)^2d\sigma(x),\quad v\ge 0.
\end{equation}
In particular, $u^j$ is the solution of
\begin{equation}\label{eq-main-av-j}
-\Delta u^j = \cpt_{\nu(x)}(T)(\psi-u)^+d\sigma|_{\Gamma^j} + f\chi_{\{u^j>0\}}.
\end{equation}
That is, $\cpt_{\nu(x)}(T)$ depends in general on the point $x\in\Gamma^j$.
\end{lemma}

\begin{proof}
To indicate the dependence of the correctors on $\Gamma^j$, we write
$w_{\e,j}$ for the corrector and $\mu_{\e,j}$ for the corresponding measure, given by \eqref{cor-10}-\eqref{cor-mu}.
\[
F_\e(x) = \frac{1}{\sigma(Q_{\e^q}(x)\cap\Gamma^j)}\int_{Q_{\e^q}(x)}|\nabla w_{\e,j}|^2dy
\]
converges to $F(x)=\cpt_{\nu(x)}(T)$ as $\e\to0$, by Lemma~\ref{corr.energy}.
Furthermore this convergence is uniform by definition of $\Gamma^j$.
The generalization of Lemma~\ref{compact embedding} to a hyper-surface of class $C^2$ is strainght forward.
It thus remains only to determine the weak limit of $\mu_{\e,j}$.
For any $\e>0$, we may cover $\Gamma^j$ by a finite number of disjoint sets
\[
Q_{\e^p}(x^i_\e)\cap\Gamma^j,\quad x^i_\e\in \Gamma^j.
\]
Thus, if $E\subset \R^n$, then
\begin{align*}
\mu_j(E)&=\lim_{\e\to0}\int_{E}|\nabla w_{\e,j}|^2dx=
\lim_{\e\to0}\sum_i \sigma(Q_{\e^p}(x^i_\e)\cap\Gamma^j)F_\e(x^i_\e)\\
&=\int_{\Gamma \cap E}\cpt_{\nu(x)}(T)d\sigma(x),
\end{align*}
due to the uniform convergence of $F_\e$.
The result follows exactly as in the hyper-plane case, but equation \eqref{lapl.corr} is replaced by
\begin{equation}\label{lapl.corrj}
\lim_{\e\to0}\l\Delta w_{\e,j}, \phi z_\e\r \to \int_\Gamma^j\cpt_{\nu(x)}(T) \phi zd\sigma(x).
\end{equation}
\end{proof}

\begin{thm} \label{thm-surf}
Let $u_\e$ solve \eqref{eq-main-e} and suppose $\Gamma$ satisfies assumption \eqref{smallcap}.
Then as $\e\to0$, $u_\e\weak u$ in $H_0^1(\Omega)$ where $u$ is the unique minimizer of
\begin{equation}\label{functional-av}
J_\nu(v):= \int_\Omega\frac12|\nabla v|^2 -fvdx + \frac12\int_\Gamma\cpt_{\nu(x)}(T)((\psi-v)^+)^2d\sigma(x),\quad v\ge 0.
\end{equation}
In particular, $u$ is the solution of
\begin{equation}\label{eq-main-av}
-\Delta u = \cpt_{\nu(x)}(T)(\psi-u)^+d\sigma + f\chi_{\{u>0\}}.
\end{equation}
\end{thm}

\begin{proof}
Let $C=\|\psi\|_{L^\infty}$ and let $v^j$ be the capacity potential of $\Gamma\setminus\Gamma^j$.
Then $-\Delta v^j\ge 0$ and $v^j\ge\chi_{\Gamma\setminus\Gamma^j}$. By our assumption, \eqref{smallcap},
$v^j\weak 0$ in $H_0^1(\Omega)$. Since $\psi_\e^j=\psi\chi_{\Gamma^j_\e}\le \psi\chi_{\Gamma_\e}=\psi_\e$
we have $u^j_\e\le u_\e$. Let $g^j_\e=u^j_\e+Cv^j$. Then $g_\e^j\ge\psi_\e$ and $-\Delta g_\e^j\ge 0$ in $\Omega$.
Therefore $g^j_\e\ge u_\e$ in $\Omega$. Indeed,
otherwise $\min\{g^j_\e,u_\e\}$ would be a supersolution of \eqref{eq-main-e} that is smaller than
$u_\e$ on some set of positive capacity, contradicting the minimality of $u_\e$. Thus
\[
u^j_\e\le u_\e\le u^j_\e+v^j.
\]
Taking first $\e\to0$, then $j\to\infty$, we see that
\[
u^j\weak u.
\]
Since
\begin{align}\label{anothereq}
&\int_\Omega\frac12|\nabla u^j|^2 +fu^jdx + \frac12\int_{\Gamma^j}\cpt_{\nu(x)}(T)((\psi-u^j)^+)^2d\sigma(x)\\
&\le \int_\Omega\frac12|\nabla v|^2 +fvdx + \frac12\int_{\Gamma^j}\cpt_{\nu(x)}(T)((\psi-v)^+)^2d\sigma(x),\\
&\text{ for all }v\in H_0^1(\Omega),\; v\ge 0,
\end{align}
the conclusion follows after passing to the limit $j\to\infty$ in \eqref{anothereq},
and using the weak lower semicontinuity of the norm on $H_0^1(\Omega)$ for the term
$\int_\Omega\frac12|\nabla u^j|^2dx$.

\end{proof}

We conclude by giving an example of a hyper-surface $\Gamma$ that satisfies \eqref{smallcap}.

\begin{example} \label{ex-cylinder}
Let $\theta$ be a real number such that $\tan{\theta}\in\cA$ and set $e=(\cos{\theta}, \sin{\theta}, 0, \cdots,0)\in\BR^n$ and
$e^\perp=(\sin{\theta},-\cos{\theta})$.
Let $g$ be a smooth real valued function depending on the $n-2$ variables $x_3,\ldots, x_n$ and define the graph
\[
\mathcal{G} = \{ g(x_3,\ldots,x_n)e + (0,0,x_3,\ldots,x_n)\}.
\]
Then a hyper-surface in $\BR^n$ is constructed by
\[
\Gamma=\mathcal{G}+te^\perp,\quad t\in\BR.
\]
The normal vector $\nu(x)$ to $\Gamma$ at $x$ is always orthogonal to $e^\perp$ and thus
$\nu_2(x)/\nu_1(x)=\tan\theta\in\cA$ for all $x\in\Gamma$.
This means that $\Gamma=\Gamma^j$ when $j$ is large enough, which clearly implies \eqref{smallcap}.

\end{example}

\section{Appendix: Uniform distribution mod 1} \label{Uniform distribution}

This section contains a general discussion of uniform distribution mod 1, and builds up
the necessary theory for the present homogenization problem.
An excellent introduction to the theory of uniform distribution is the book by
Kuipers and Niederreiter, \cite{MR0419394}.
For the readers convenience we have gathered the basic theory of uniform distribution, in particular discrepancy,
here. Most of the material, except possibly Lemma~\ref{lemma-equidist}, is standard.

A fundamental problem already encountered in the introduction was that of
estimating the error in the approximation
\[
\frac{A_\e}{N_\e}\approx \frac{d_\e-c_\e}{\e},
\]
where $N_\e$ is the number of $k'\in \BZ^{n-1}$ such that
\[
(k',k'\cdot\alpha)\in \e^{-1}\Omega,
\]
and $A_\e$ the fraction of these points that intersect $\cT_\e$,
see \eqref{eq-o4}.
Thus we are led to study the distribution mod 1 of sequences of this
type. We start by considering sequences $k\alpha$ with $k\in\BZ$ and
$\alpha\in\R$. This will then be generalized to the higher
dimensional case.

\subsection{Known Result for the 1-dimensional sequence $\{k \alpha\}$} \label{sec-unif-1}

First we define the notion of uniform distribution.
\begin{definition}[Uniform distribution mod 1] \item \begin{enumerate}
\item
Let $\{x_j\}_{j=1}^\infty$ be given sequence of real numbers. For a positive integer $N$ and a subset $E$ of $[0,1]$, let the
counting function $A\left( E;\{x_j\};N \right)$ be defined as the numbers of terms $x_j$, $1 \le j \le N$, for which $x_j \in E$ (mod 1).
\item
The sequence of real numbers $\{x_j\}$ is said to be uniformly distributed modulo 1 if for every
pair $a, b$ of real numbers with $0 \le a < b \le 1$ we have
\begin{equation}
\lim_{N \ra \infty} \displaystyle\frac{A \left( [a,b);\{x_j\};N\right)}{N} = b-a.
\end{equation}
\end{enumerate} \end{definition}

Since $\int_{[0,1]} \chi_{[a,b)} dx = b-a$, we can deduce
\begin{equation}
\lim_{N \ra \infty} \frac{1}{N}\sum_{n=1}^N \chi_{[a,b)}(x_n) = \int_{[0,1]} \chi_{[a,b)} dx,
\end{equation}
where we have extended $\chi_{[a,b)}$ to a 1-periodic function in $\R$.

This simple observation and approximation technique lead to the following criterion.
\begin{thm}
The sequence $\{x_j\}_{j=1}^\infty$ is uniformly distributed mod $1$ if and only if for every real-valued continuous
function $f$ defined on the closed unit interval $I=[0,1]$, we have (upon extending $f$ periodically to $\R$)
\begin{equation}
\lim_{N \ra \infty} \frac{1}{N}\sum_{n=1}^N f(x_n) = \int_{[0,1]}f dx.
\end{equation}
\end{thm}
Fourier expansion of the function $f$ above leads to a useful
criterion for uniform distribution.
\begin{lemma}[Weyl's criterion]
A sequence $\{x_j\}_{j=1}^\infty$ is uniformly distributed mod 1
if and only if
\[
\frac{1}{N}\sum_{j=1}^Ne^{2\pi ilx_j}\to0,\;\text{as }N\to\infty,
\]
for any nonzero $l\in\Z$.
\end{lemma}

Let us start by assuming that $\alpha\in\R$ and $k\in\mathbb{N}$.
Then it follows from Weyl's criterion that the sequence
$\{k\alpha\}_k$ is uniformly distributed if and only if $\alpha$ is
irrational.

In fact, it turns out that we will need more information than "the
sequence $\{k\alpha\}_k$ has uniform distribution". If a sequence
$\{x_j\}$ is uniformly distributed the above definition implies that
\[
A([a,b);\{x_j\},N)= (b-a)N(1 + \rho(N^{-1})),
\]
where $\rho(0^+)=0$. However, if, for example $a=0$ and $b=N^{-\frac12}$, we cannot assert
that
\begin{equation} \label{ex-1}
A([0,N^{-\frac12});\{x_j\},N)= N^{-\frac12}N(1 + \rho(N^{-1})),\quad \rho(0^+)=0.
\end{equation}
This is a much stronger result that relies on deeper arithmetic properties of
of $\alpha$.

To proceed, we recall the discrepancy of a sequence
$\{x_j\}_{j=1}^{\infty}$.
\begin{definition}\label{def_discrepancy}
Let $\{x_j\}_{j=1}^\infty$ be a sequence of real numbers. The discrepancy of its $N$ first elements is the number
\[
D_N(\{x_j\}_{j=1}^N) = \sup_{0\le a<b\le 1}
\left|\frac{A([a,b);\{k\alpha\},N)}{N}-(b-a)\right|.
\]
If $x_j=j\alpha$, we simply write $D_N(\alpha)$.
\end{definition}
Note that
\begin{align*}
&|A([0,N^{-\frac12});\{x_j\},N)- N^{-\frac12}N|= N\left|\frac{A([0,N^{-\frac12});\{x_j\},N)}{N}- N^{-\frac12}\right|\\
&\le ND_N(\alpha),
\end{align*}
by definition of the discrepancy. Thus if $D_N=o(N^{-\frac12})$ we
obtain \eqref{ex-1}. The discrepancy of sequences $\{k\alpha\}_{k=1}^\infty$,
$\alpha\in\BR$ has been studied extensively. Here are some strong
results:

\begin{thm}\cite{MR0168546}\label{kesten}
\begin{equation}
\frac{ND_N(\alpha)}{\log N\log(\log N)}\to \frac{2}{\pi^2}
\end{equation}
in measure w.r.t. $\alpha$ as $N\to\infty$. In particular, this result is true for a.e. $\alpha$ in a bounded set.
\end{thm}
\begin{thm}\cite{MR0419394}\label{discest}
For a.e. $\alpha\in\BR$ holds
\[
D_N(\alpha)=O\left(\frac{\log^{2+\delta}N}{N}\right),
\]
for any $\delta>0$.
\end{thm}
Theorem ~\ref{kesten} is due to Harry Kesten, \cite{MR0168546}, and Theorem
~\ref{discest} can be found as an exercise in \cite{MR0419394}. The importance of these theorems
to the application at hand is that
\begin{equation}\label{ud:5}
D_N(\alpha)=o\left(\frac{1}{N^p}\right),\quad\text{for and any }  0<p<1.
\end{equation}
In view of Theorem~\ref{discest}, the condition \eqref{ud:5} holds for a.e. $\alpha\in\R$.

\begin{definition}
If $\alpha\in\R$ satisfies the condition of \eqref{ud:5}
we write
\[
\alpha\in \cA.
\]

\end{definition}

\subsection{Application to the sequence $\{ k' \cdot \alpha \}$ } \label{sec-unif-2}

Let us describe how to apply the theory of uniform distribution mod 1 to our homogenization problem.
Let $\Gamma=\{x\cdot\nu=x_0\cdot\nu\}$. Then the limiting energy of the correctors,
\[
\lim_{\e\to0}\int_\Omega|\nabla w_\e|^2dx,
\]
determines the character of the limit problem, as described in the outline.
As soon as we know how many times $\Gamma$ intersects a certain portion of $\cT_\e$,
this energy may be computed. By assuming $\nu_n\neq0$, $\Gamma$ can be represented by a graph of an affine function as follows:
\[
\Gamma=\{(x', { \alpha}\cdot x'+c)\},
\]
where
\begin{align}
&\alpha_i=-\nu_i/\nu_n \text{ for } i=1,\ldots,n-1,\label{hp1}\\
&c=x_0\cdot\nu/\nu_n,\label{hp2}\\
&x'\in\R^{n-1}.\nonumber
\end{align}
We assume $c=0$, but see Remark~\ref{rem-translation}.

To make our problem more precise, we introduce a few notations first. Let $\Gamma'_\Omega$ be the projection of
$\Gamma \cap \Omega$ on $\R^{n-1} \times \{0\}$, i.e.
\begin{equation}
\Gamma'_\Omega = \{ x' \in \R^{n-1} : (x',x_n) \in \Omega\cap\Gamma \text{ for some } x_n \in \R \}
\end{equation}
and let 
\begin{equation}
\cZ_\e = \e^{-1} \Gamma'_\Omega \cap \Z^{n-1}.
\end{equation}

As we saw in the introduction, for any fixed $k' \in \cZ_\e$, there exists a unique $k_n \in \Z$ such that
\begin{equation}\label{k,k'}
\alpha \cdot k' -k_n \in \left[ -\frac{1}{2},\frac{1}{2} \right).
\end{equation}
We shall only consider $k=(k',k_n) \in \e^{-1} \Omega \cap \Z^n$ such that \eqref{k,k'} holds and  $k'\in \cZ_\e$.
Then the number $\alpha \cdot k' -k_n$ determines where $\Gamma$ intersects $Q_\e(\e k)$.
Since the perforations
\[
T^k_\e= a_\e T+\e k, \quad a_\e=\e^{\frac{n}{n-1}}
\]
have decay rate $a_\e$, we have to determine the number of points
$\alpha\cdot k'- k_n$ in any interval of length proportional to
\begin{equation}\label{mod1perforation}
a_\e/\e=\e^p,\quad p=1/(n-1).
\end{equation}
We recall that, for $[t ,t+\e^p)\subset [-1/2,1/2)$,
\[
\alpha\cdot k'- k_n\in [t ,t+\e^p)\text{ if and only if }
\alpha\cdot k'/\BZ\in[t ,t+\e^p)/\BZ.
\]
For this reason we define
\begin{align}
&N(\e):=\#\cZ_\e = \#\left(\e^{-1} \Gamma'_\Omega \cap \Z^{n-1} \right)\label{lem-freq1},\\
&A(\e^p,t):=\#\{k' \in \e^{-1} \Gamma'_\Omega \cap \Z^{n-1} :\alpha\cdot k'/\BZ \in [t ,t+\e^p) / \Z \} \label{lem-freq2}.
\end{align}
Note that $A(\e^p,t)$ and $N(\e)$ 
depend on $\alpha$ and the set $\Gamma_\Omega'$.
Our aim is to prove that whenever some $\alpha_i\in\cA$,
\[
A(\e^p,t)=(1+\rho(\e))N(\e)\e^p, \quad\rho(0+)=0,
\]
for some modulus of continuity $\rho(\e)$ that is independent of $t$.
\begin{remark}\label{rem-translation}
If $c\neq 0$ in \eqref{hp2} we get
\[
A(\e^p,t)=\#\{k'\in\e^{-1}\Gamma_\Omega'\cap\BZ^{n-1}: (\alpha k'+c/\e)/\BZ\in (t,t+\e^p)/\BZ\}.
\]
But if this is asymptotically independent of $t$ for $c=0$, then the same holds for any $c$. The set
$\Gamma_\Omega'$ will however depend on $c$.
\end{remark}

For notational convenience we prove the next lemma for a set $E\subset\R^m$, replacing
$\Gamma_\Omega'$ by $E$ in \eqref{lem-freq1} and \eqref{lem-freq2}.
The result then follows
by taking $E=\Gamma_\Omega'$, $m=n-1$.
We remark that Lemma~\ref{lemma-equidist} below could possibly be found in the litterature, but we have not been able to retrieve it.
However, Step 1 is essentially a consequence of Theorem 2.6 in \cite{MR0419394}.
\begin{lemma}\label{lemma-equidist}
Let $E\subset\R^m$ be Lebesgue measurable and have positive measure. Let $\alpha=(\alpha_1,\ldots,\alpha_m)$
and assume $\alpha_i\in\cA$, for at least one $i\in\{1,\ldots,m\}$.
Let
\begin{align}
&N(\e) = \#\left(\e^{-1} E \cap \Z^{n-1} \right)\label{lem-freq1E},\\
&A(\e^p,t):=\#\{k' \in \e^{-1}E \cap \Z^{n-1} :\alpha\cdot k'/\BZ \in [t ,t+\e^p) / \Z \} \label{lem-freq2E}.
\end{align}
Then for any $0<p<1$,
\begin{align*}
A(\e^p,t)=(1+\rho(\e))N(\e)\e^p, \text{ for some }\rho \text{ such that }\rho(0^+)=0.
\end{align*}
\end{lemma}

To prove the lemma we will use a Fubini-type summation argument and
the classical result concerning the distribution mod 1
of sequences of the type $\{k\alpha\}_{k\in\BN}$, \eqref{ud:5}.
The idea is to control $E(\e^p)$ in \eqref{lem-freq3} by a quantity of the type
\[
D_{\e^{-1}}(\alpha^i),
\]
where $D_{\e^{-1}}(\alpha^i)$ is the discrepancy of the sequence
$\{j\alpha^i\}$, $1\le j\le \e^{-1}$, and $\alpha^i\in\cA$.

\begin{proof}[Proof of Lemma~\ref{lemma-equidist}]
Step 1. Suppose first that $E$ is a cube, $E=x+(a,b)^m$.
Without loss of generality we may assume $\alpha_m\in\cA$. Let
\begin{equation}\label{eq-d01}
S_\e':= \{k'\in\Z^{m-1}:(k',k_m)\in\e^{-1}E,\;\text{ for some }k_m\in\Z\}.
\end{equation}
If $k'\in S_\e'$, there exist integers $m_\e,M_\e$
such that
\begin{equation}\label{eq-d02}
(k',k_m)\in (\e^{-1}E)\cap\Z^m,\quad\text{for }m_\e\le k_m \le M_\e.
\end{equation}
Hence
\[
N(\e)=(\#S_\e')H_\e,\quad H_\e=M_\e-m_\e.
\]
Let
\[
A_{k'}(\e^p,t) = \# \{ k_m : \alpha' \cdot k' + \alpha^m k_m \in [t, t+\e^p] /\Z, m_\e \le k_m \le M_\e \},
\]
where $\alpha'=(\alpha_1,\ldots,\alpha_{m-1})$.
Then
\[
A(\e^p,t)=\sum_{k'\in S_\e'}A_{k'}(\e^p,t).
\]
From \eqref{eq-d02} we conclude that
\begin{align*}
&A_{k'}(\e^p,t)\\
&=\#\{k_m : \alpha' \cdot k' + \alpha_m k_m \in [t, t+\e^p] /\Z, m_\e \le k_m \le M_\e \} \\
&=\#\{k_m:\alpha'\cdot k'+(m_\e-1)\alpha_m+k_m\alpha_m\in [t, t+\e^p] /\Z,\;1\le k_m \le H_\e+1\}\\
&=\#\{k_m:k_m\alpha_m\in [\tilde{t}, \tilde{t}+\e^p] /\Z,\;1\le k_m \le H_\e+1\}
\end{align*}

for some $\tilde{t}$ where $H_\e=M_\e-m_\e$.
So, for fixed $k' \in S_\e'$ we can apply \eqref{ud:5} to the sequence
\begin{equation*}
\{k_m:k_m\alpha_m\in [\tilde{t}, \tilde{t}+\e^p] /\Z,\;1\le k_m \le H_\e\}.
\end{equation*}
Hence we have the following estimate which is uniform in $t$ and $\tilde{t}$:
\begin{align*}
&\left| \displaystyle\frac{A_{k'}(\e^p,t))}{H_\e} -  \left| [\tilde{t}, \tilde{t}+\e^p]  \right| \right| \\
&=\left| \displaystyle\frac{A_{k'}(\e^p,t)}{H_\e} -  \e^p \right|  \le D_{H_\e} (\alpha_m) = o\left(H_\e^{-p}\right),\quad 0<p<1.
\end{align*}

From the above estimate, we have
\begin{equation} \begin{aligned}
\left|\frac{A(\e^p,t)}{N(\e)}-\e^p\right| &\le \frac{H_\e}{N_\e}\sum_{k'\in S_\e'}
\left|\frac{A_{k'}(\e^p,t)}{H_\e}-\e^p\right|\\
&\le \frac{H_\e}{N_\e}\sum_{k'\in S_\e'}D_{H_\e}(\alpha^m)=D_{H_\e}(\alpha^m)
=o(\e^{p}),\;\text{ as }\e\to 0.
\end{aligned} \end{equation}


Step 2.
Suppose $E$ is the union of a finite number of disjoint cubes,
\[
E=\bigcup_{j=1}^MQ^j,\quad Q^j\cap Q^i=\emptyset\text{ if }i\neq j.
\]
Let
\begin{align*}
&N_j(\e) = \#\left(\e^{-1} Q^j \cap \Z^{n-1} \right),\\
&A_j(\e^p,t):=\#\{k' \in \e^{-1}Q^j \cap \Z^{n-1} :\alpha\cdot k'/\BZ \in [t ,t+\e^p) / \Z \}.
\end{align*}
Then
\begin{align*}
&N(\e) = \sum_{j} N_j(\e),\\
&A(\e^p,t) = \sum_j A_j(\e^p,t).
\end{align*}
Hence we obtain
\begin{equation*}
\left| \displaystyle\frac{A(\e^p,t)}{N(\e)} - \e^p \right| = \sum_{j} \displaystyle\frac{N_j(\e)}{N(\e)}
\left| \displaystyle\frac{A_j(\e^p,t)}{N_j(\e)} - \e^p \right|.
\end{equation*}

By Step 1. each term in this sum is $o(\e^{p})$.

Step 3.
For the general case, the fact that the Lebesgue measure is regular allows us, for any given $\delta>0$,
to choose domains $L$ and $U$ consisting of a finite
number of disjoint cubes such that $L\subset E \subset U$ and
$|U-L|<\delta$.
Let $A_L(\e^p,t)$, $N_L(\e)$ and $A_U(\e^p,t)$, $N_U(\e)$ be given by \eqref{lem-freq1E}-\eqref{lem-freq2E}, with $E$ replaced by
$L$ and $U$ respectively.
Then, from the relation $L \subset E \subset U$, we have
\begin{equation*}
\e^{-p} \left( \displaystyle\frac{A_L(\e^p, t)}{N_U(\e)} -\e^p \right) \le \e^{-p} \left( \displaystyle\frac{A(\e^p, t)}{N(\e)} -\e^p \right) \le \e^{-p} \left( \displaystyle\frac{A_U(\e^p, t)}{N_L(\e)} -\e^p \right).
\end{equation*}
The third term of above can be written as
\begin{equation*}
\e^{-p} \left( \displaystyle\frac{A_U(\e^p, t)}{N_L(\e)} -\e^p \right) =  \e^{-p} \displaystyle\frac{N_U}{N_L} \left( \displaystyle\frac{A_U(\e^p, t)}{N_U(\e)} -\e^p \right) + \left(\displaystyle\frac{N_U}{N_L} -1 \right).
\end{equation*} 
Since $\lim \displaystyle\frac{N_U}{N_L} = \displaystyle\frac{|U|}{|L|}$, we have
\begin{equation*} \begin{aligned}
&\limsup_{\e \ra 0} \e^{-p} \left( \displaystyle\frac{A(\e^p, t)}{N(\e)} -\e^p \right) \\
&\le \limsup_{\e \ra 0} \e^{-p} \displaystyle\frac{N_U}{N_L} \left( \displaystyle\frac{A_U(\e^p, t)}{N_U(\e)} -\e^p \right) +  \left( \displaystyle\frac{|U|}{|L|} -1 \right) \\
&\le (1+2\delta) \lim \left( \displaystyle\frac{A_U(\e^p, t)}{N_U(\e)} -\e^p \right) + \displaystyle\frac{\delta}{|L|} \\
&\le \displaystyle\frac{\delta}{|L|}
\end{aligned} \end{equation*}
Similarly, we have
\begin{equation*}
\liminf_{\e \ra 0} \e^{-p} \left( \displaystyle\frac{A(\e^p, t)}{N(\e)} -\e^p \right) \ge -\displaystyle\frac{\delta}{|L|},
\end{equation*}
and the convergence is uniform with respect to $t$. This completes the proof.

\end{proof}

For a plane $\Gamma=\{x\cdot\nu=c\}$ we have shown that if $\nu_j\neq0$,
$\alpha=(\alpha_1,\ldots,\alpha_{n-1})$, $\alpha_i=\nu_i/\nu_j$, $i\neq j$ and if
$\alpha_i\in\cA$ for some $i$, then
\[
\frac{A(\e^p,t)}{N(\e)}=\e^p+o(\e^p),\quad 0<p<1.
\]
It is clearly desirable to determine the set of directions $\nu\in S^{n-1}$ for which this holds.
\begin{lemma}
For a.e. $\nu\in S^{n-1}$, there is a pair $\nu_i,\nu_j$ such that $\nu_i/\nu_j\in \cA$.

\end{lemma}\label{aenormal}
\begin{proof}
Fix a set $E_j \subset S^{n-1}$ and suppose $\nu_j\ge \lambda >0$ on $E_j$.
Let $\tilde E_j =\{\nu\in E_j :\nu_i/\nu_j\not\in \cA \text{ for }i\neq j\}$.
If we prove $m_{S^{n-1}}(\tilde E_j)=0$ we are done, since $S^{n-1}$ may be covered by sets
$\{ E_j \}_{1\le j\le n}$ such that $\nu_j>\lambda$ on $E_j$, for some small $\lambda$.
Let $\Phi:S^{n-1}\to \R^{n-1}$ be a local diffeomorphism such that
\begin{align*}
m_{S^{n-1}}(\tilde E_j)&=\int_{\Phi(\tilde E_j)}A(\Phi^{-1}(u))\text{det}D\Phi^{-1}(u)du,\\ 
&du\text{ Lebesgue measure on }\R^{n-1}.
\end{align*}
Define a new diffeomorphism $\Psi:E_j\to \R^{n-1}$ by $\Psi(\nu)=(\nu_i/\nu_j)_{i\neq j}$.
Then
\begin{align*}
&\int_{\Phi(\tilde E_j)}A(\Phi^{-1}(u))\text{det}D\Phi^{-1}(u)du\\
&=\int_{\Psi(\tilde E_j)\subset \{v:v_i\not \in {\cA},\;i\neq j\}}A(\Psi^{-1}(v))|\text{det}D\Psi^{-1}|dv=0,
\end{align*}
since meas$( \{v_i\in\R:v_i\not \in {\cA} \})=0$ for any $i \neq j$.

\end{proof}



\end{document}